\documentclass[11pt,reqno]{amsart}
\usepackage{amsmath, amssymb, amsthm}
\usepackage{url}
\usepackage[breaklinks]{hyperref}

 \renewcommand{\a}{\alpha}
\renewcommand{\b}{\beta}

\newcommand{\g}{\gamma}

\renewcommand{\k}{\kappa}

\renewcommand{\t}{\tau}

\renewcommand{\(}{\left\(}
\renewcommand{\)}{\right\)}
\renewcommand{\[}{\left\[}
\renewcommand{\]}{\right\]}
\renewcommand{\i}{\infty}
\numberwithin{equation}{section}
 \theoremstyle{plain}
\newtheorem{theorem}{Theorem}[section]
\newtheorem{lemma}[theorem]{Lemma}
\newtheorem{remark}[]{Remark}

\newtheorem{corollary}[theorem]{Corollary}

   \makeatletter
\def\proof{\@ifnextchar[{\@oproof}{\@nproof}}
\def\@oproof[#1][#2]{\trivlist\item[\hskip\labelsep\textit{#2 Proof of\
#1.}~]\ignorespaces}
\def\@nproof{\trivlist\item[\hskip\labelsep\textit{Proof.}~]\ignorespaces}

\makeatother

\begin{document}
\title[One variable Generalization of five entries of Ramanujan]{One variable Generalization of five entries of Ramanujan and their finite analogue}

\thanks{$2020$ \textit{Mathematics Subject Classification.} Primary 11P81,  33D15; Secondary  11P84, 05A17.} 

\author{Archit Agarwal}
\address{Archit Agarwal, Department of Mathematics, Indian Institute of Technology Indore, Simrol, Indore 453552, Madhya Pradesh, India.}
\email{phd2001241002@iiti.ac.in, archit.agrw@gmail.com}

\maketitle

\begin{abstract}
Ramanujan recorded five $q$-series identities at the end of his second notebook and an unified generalization of these identities obtained by Bhoria, Eyyunni and Maji. Recently, Dixit and Patel gave a finite analogue of the identity of Bhoria et. al.  which in turn gives finite analogue of all the aforementioned identities of Ramanujan. In this paper, one of our main goals is to obtain a one-variable generalization of the identity of Bhoria et. al. along with its finite analogue, which naturally generalizes the result of Dixit and Patel. Utilizing these newly established identities, we derive one-variable generalizations for each of the five entries by Ramanujan and their corresponding finite analogues.

\end{abstract}

\section{Introduction}
Ramanujan catalogued many $q$-series identities in Chapter XVI of his second notebook  \cite{ramanujantifr}.  He \cite[pp.~302--303]{ramanujantifr} further listed five $q$-series identities at the end of his second notebook.
Proofs of these five identities can be found in \cite[pp.~262--265]{bcbramforthnote}. In order to state these identities, we first introduce certain essential notations. For any complex number $q$ with $|q| < 1$, the $q$-Pochhammer symbol is defined as
\begin{align*}
&(a)_0 := (a;q)_0 = 1,\\
&(a)_n := (a;q)_n = (1-a)(1-aq)\cdots(1-aq^{n-1}),~n\geq1,\\
&(a)_\infty := (a;q)_\infty := \lim_{n\rightarrow \infty}(a;q)_n,~\textrm{for}~|q|<1.
\end{align*}
These notations are fundamental to the theory of $q$-series. Below, we present the five identities given by Ramanujan.

\textbf{Entry 1:} Let $n \in \mathbb{N}$, $a \neq 0, b \neq q^{-n}$. Then 
\begin{align}\label{entry 1}
\frac{(-a q)_{\infty}}{(b q)_{\infty}}=\sum_{n=0}^{\infty} \frac{(-b / a)_{n} a^{n} q^{n(n+1) / 2}}{(q)_{n}(b q)_{n}}.
\end{align}
Ona can also find the above identity in the Lost Notebook \cite[p.~370]{lnb}.

\textbf{Entry 2:} Let $a \neq q^{-n}$ for any natural number $n$, then one has
\begin{align}\label{entry 2}
(a q)_{\infty} \sum_{n=1}^{\infty} \frac{n a^{n} q^{n^{2}}}{(q)_{n}(a q)_{n}}=\sum_{n=1}^{\infty} \frac{(-1)^{n-1} a^{n} q^{n(n+1) / 2}}{1-q^{n}}.
\end{align}

\textbf{Entry 3:} For $a \neq 0$, and $b \neq q^{-n}, n \geq 0$, the following identity is true:
\begin{align}\label{entry 3}
\sum_{n=1}^{\infty} \frac{(b / a)_{n} a^{n}}{\left(1-q^{n}\right)(b)_{n}}=\sum_{n=1}^{\infty} \frac{a^{n}-b^{n}}{1-q^{n}}.
\end{align}
Letting $a \rightarrow 0$, and replacing $b$ by $zq$ gives the next identity.

\textbf{Entry 4:} For $z\neq q^{-n}$ for any $n \geq 1$, we have
\begin{align}\label{entry 4}
\sum_{n=1}^{\infty} \frac{(-1)^{n-1} z^{n} q^{\frac{n(n+1)}{2}}}{\left(1-q^{n}\right)(z q)_{n}}=\sum_{n=1}^{\infty} \frac{z^{n} q^{n}}{1-q^{n}}.
\end{align}
The above identity was rediscovered by Uchimura \cite[Equation (3)]{uchimura81} and Garvan \cite{garvan1}. The last identity is as follows:

\textbf{Entry 5:} If $a \neq q^{-n}$ for any $n \geq 0$, then
\begin{align}\label{entry 5}
\sum_{n=1}^{\infty} \frac{a^{n}(q)_{n-1}}{\left(1-q^{n}\right)(a)_{n}}=\sum_{n=1}^{\infty} \frac{n a^{n}}{1-q^{n}}.
\end{align}
 In 1919, Kluyver \cite{kluyver} found a special case of \eqref{entry 4}, namely, for $z=1$,
\begin{align}\label{kluyver's identity}
\sum_{n=1}^{\infty} \frac{(-1)^{n-1} q^{\frac{n(n+1)}{2}}}{\left(1-q^{n}\right)(q)_{n}}=\sum_{n=1}^{\infty} \frac{q^{n}}{1-q^{n}}.
\end{align}
Fine \cite[p.~14, Equation~(12.4),~(12.42)]{fine}, Uchimura \cite[Theorem 2]{uchimura81} and Zudilin \cite[P.~4]{zudilin} also rediscovered \eqref{kluyver's identity}. Uchimura gave a new expression to \eqref{kluyver's identity}, mainly, he showed that
\begin{align*}
\sum_{n=1}^\infty n q^n (q^{n+1})_\infty=\sum_{n=1}^{\infty} \frac{(-1)^{n-1} q^{\frac{n(n+1)}{2}}}{\left(1-q^{n}\right)(q)_{n}}=\sum_{n=1}^{\infty} \frac{q^{n}}{1-q^{n}}.
\end{align*}
Over the years, this identity has gained significant attention from numerous mathematicians and has been generalized in various ways. Readers interested in exploring these generalizations can see \cite{ABEM23}--\cite{ABEM24} and references therein.

In 2020, Dixit and Maji \cite[Theorem 2.1]{DM} obtained a one variable generalization of \eqref{entry 3}, namely, for complex numbers $a,~b,~c$ with $|a|<1$ and $|cq|<1$, they proved that
\begin{align}\label{dixit maji}
\sum_{n=1}^\infty \frac{(b/a)_na^n}{(1-cq^n)(b)_n}=\sum_{n=0}^\infty \frac{(b/c)_n c^n}{(b)_n} \left(\frac{aq^n}{1-aq^n}-\frac{bq^n}{1-bq^n}\right).
\end{align}
In the same paper, they also derived a more general form of \eqref{entry 4} utilizing \eqref{dixit maji}. Namely, for $|cq|<1$,
\begin{align}\label{corollary of dixit maji}
\sum_{n=1}^\infty \frac{(-1)^{n-1}z^nq^{\frac{n(n+1)}{2}}}{(1-cq^n)(zq)_n}=\frac{z}{c}\sum_{n=1}^\infty \frac{(zq/c)_{n-1}}{(zq)_n}(cq)^n.
\end{align}
The above identity is also a one variable generalization of Andrews, Garvan and Liang's identity \cite[Theorem 3.5]{agl13}. With the help of the aforementioned identity, Dixit and Maji \cite{DM} gave a natural proof of the following identity of Garvan \cite[Equation 1.3]{garvan1}, that is for $|z|\leq 1$,  one has
\begin{align}\label{garvan}
\sum_{n=1}^\infty \frac{(-1)^{n-1} z^n q^{n^2}}{(zq;q^2)_n(1-zq^{2n})}=\sum_{n=1}^\infty \frac{(q;q)_{n-1} z^n q^{\frac{n(n+1)}{2}}}{(zq;q)_n}.
\end{align}
In the same paper, Dixit and Maji further explored many beautiful partition theoretic interpretations arising from \eqref{dixit maji} and \eqref{corollary of dixit maji}. Interested readers can go through \cite{DM} for more details.

Later, Dixit, Eyyunni, Maji and Sood \cite[Theorem 1.1]{DEMS} obtained a finite analogue of the identity \eqref{dixit maji}. For any natural number $N$, complex numbers $a,~b,~c$ with $a,b,c\neq q^{-n},1\leq n\leq N-1, c\neq q^{-N}$, and $a,b\neq 1$, they proved that
\begin{align}\label{finite analogue of dixit maji}
&\sum_{n=1}^N \begin{bmatrix}
N\\n
\end{bmatrix} \frac{(a)_{N-n}(b/a)_n (q)_n a^n}{(a)_N (1-cq^n)(b)_n}\nonumber\\&=\sum_{n=1}^N \begin{bmatrix}
N\\n
\end{bmatrix} \frac{(cq)_{N-n} (b/c)_{n-1} (q)_n c^{n-1}}{(cq)_N (b)_{n-1}}\left( \frac{aq^{n-1}}{1-aq^{n-1}}-\frac{bq^{n-1}}{1-bq^{n-1}} \right),
\end{align}
where
\begin{align*}
\begin{bmatrix}
N\\n
\end{bmatrix} = \begin{bmatrix}
N\\n
\end{bmatrix}_q :=
\begin{cases} 
\frac{(q;q)_N}{(q;q)_n (q;q)_{N-n}},&~~ \textrm{if}~ 0\leq n\leq N,\\
\hspace{9mm}0,&~~ \textrm{otherwise},
\end{cases}
\end{align*} is the $q$-binomial coefficient.
Further, for  $z,c \neq q^{-n},1\leq n \leq N,$ they showed that
\begin{align}\label{finite analogue of corollary of dixit maji}
\sum_{n=1}^N \begin{bmatrix}
N\\n
\end{bmatrix} \frac{ (-1)^{n-1} z^n q^{\frac{n(n+1)}{2}}(q)_n}{(1-cq^n)(zq)_n}=\frac{z}{c}\sum_{n=1}^N \begin{bmatrix}
N\\n
\end{bmatrix} \frac{(cq)_{N-n} (zq/c)_{n-1} (q)_n (cq)^{n}}{(cq)_N (zq)_{n-1}},
\end{align}
which is a finite analogue of \eqref{corollary of dixit maji}. 
With the help of \eqref{finite analogue of corollary of dixit maji}, Dixit et. al. \cite[Theorem 1.3]{DEMS} also derived a finite analogue of Garvan's identity \eqref{garvan}. Mainly, for $z\neq q^{-n}, 1\leq n\leq 4N-1$,
\begin{align}\label{finite analogue of garvan}
&\sum_{n=1}^N \begin{bmatrix}
N\\n
\end{bmatrix}_{q^2} \frac{(-1)^{n-1}(q^2;q^2)_n z^n q^{n^2}}{(zq;q^2)_n (1-zq^{2n})}\nonumber \\ &= \sum_{n=1}^N \begin{bmatrix}
N\\n
\end{bmatrix}_{q^2} \left( \frac{(q)_{2n-2} z^{2n-1} q^{n(2n-1)}}{(zq)_{2n-1}} + \frac{(q)_{2n-1} z^{2n} q^{n(2n+1)}}{(zq)_{2n}} \right)\frac{(q^2;q^2)_n}{(zq^{2N+1};q^2)_n}.
\end{align}
Many implications of \eqref{finite analogue of dixit maji} and \eqref{finite analogue of corollary of dixit maji} in the theory of restricted partitions are also studied by Dixit et. al. \cite{DEMS}.  

In a recent paper, Bhoria, Eyyunni, and Maji \cite[Theorem 2.1]{BEM21} introduced a one-variable generalization of Dixit and Maji's identity \eqref{dixit maji}. For complex numbers $a,~b,~c,~d$ with $|ad|<1$ and $|cq|<1$, one has
\begin{align}\label{bem}
\sum_{n=1}^\infty \frac{(b/a)_n(c/d)_n(ad)^n}{(b)_n(cq)_n}=\frac{(a-b)(d-c)}{(ad-b)}\sum_{n=0}^\infty\frac{(a)_n(bd/c)_nc^n}{(b)_n(ad)_n}\left(\frac{adq^n}{1-adq^n}-\frac{bq^n}{1-bq^n} \right).
\end{align}
With the help of the above identity, Bhoria et. al. were able to derived all the five identities of Ramanujan \eqref{entry 1}--\eqref{entry 5}. They also derived the following $q$-series identity of Andrews \cite[Corollary 2.2, p. 24]{yesttoday}:
\begin{align}\label{andrews q-series}
\sum_{n=1}^\infty \frac{z^n c^n q^{n^2}}{(zq)_n (cq)_n}=z \sum_{n=1}^\infty \frac{(cq)^n}{(zq)_n}.
\end{align}

Very recently, Dixit and Patel \cite[Theorem 2.1]{DixitPatel} obtained a finite analogue of the identity \eqref{bem} of Bhoria et al., which itself serves as a one variable generalization of \eqref{finite analogue of dixit maji}, namely,
\begin{align}\label{dixit patel}
&\sum_{n=1}^N \begin{bmatrix}
N\\n
\end{bmatrix} \frac{(q)_n (b/a)_n (c/d)_n (ad)_{N-n}(ad)^n}{(b)_n(cq)_n (ad)_N}\nonumber \\ &=\frac{(a-b)(d-c)}{(ad-b)}\sum_{n=1}^N \begin{bmatrix}
N\\n
\end{bmatrix} \frac{(q)_n (cq)_{N-n}(a)_{n-1}(bd/c)_{n-1}c^{n-1}}{(b)_{n-1}(ad)_{n-1}(cq)_N} \left(\frac{adq^{n-1}}{1-adq^{n-1}}-\frac{bq^{n-1}}{1-bq^{n-1}}\right).
\end{align}
Further, Dixit and Patel \cite{DixitPatel} obtained a finite analogue of all five identities of Ramanuajn \eqref{entry 1} -- \eqref{entry 5} from identity \eqref{dixit patel}. In addition, they were able to derive an identity for a finite sum of $_2\phi_1,$ that is,
\begin{align}\label{finite sum of 2phi1}
&\sum_{n=1}^N \begin{bmatrix}
N\\
n
\end{bmatrix} \frac{n(-1)^{n-1}(c/d)_nd^nq^{n(n+1)/2}}{(cq)_n}+ \frac{(c/d)_\infty (dq)_\infty}{(cq)_\infty (dq^{N+1})_\infty}\sum_{n=1}^N \begin{bmatrix}
N\\
n
\end{bmatrix}\frac{d^nq^{n(n+1)}}{(dq)_n(1-q^n)} \nonumber \\ 
&\times {}_{2}\phi_{1}\left( \begin{matrix} dq, & dq^{N+1} \\ & dq^{n+1} \end{matrix}  ;~ \frac{cq^{n+1}}{d} \right)= \frac{c}{c-d} \left(1- \frac{(dq)_N}{(cq)_N}\right)+\frac{1}{(cq)_N}\sum_{n=1}^N \begin{bmatrix}
N\\
n
\end{bmatrix} \frac{(cq/d)_n (dq)_{N-n}(dq)^n}{(1-q^n)}.
\end{align}
They also gave many applications of \eqref{finite sum of 2phi1}, especially they proved Andrew's famous identity for spt function \cite[Theorem 3]{andrews08}. They also derived many elegant $q$-series identities from \eqref{finite sum of 2phi1}. Interested readers can see \cite[Section 5]{DixitPatel} for the applications of \eqref{finite sum of 2phi1}.

This paper is arranged as follows. First, we present a one variable generalization of the identity \eqref{bem} by Bhoria et. al. along with its finite analogue, which naturally generalizes an identity \eqref{dixit patel} of Dixit and Patel and discuss some of its consequences. We also drive a one variable generalization of Garvan's identity \eqref{garvan}. In Section 3, we collect all the necessary results that will help us to prove these findings. In Section 4, we provide the proofs for the generalizations mentioned earlier.  Section 5 focuses on one-variable generalizations of Ramanujan's five identities, along with their finite analogues. In Section 6, we prove Theorem \ref{finite analogue of generalization of garvan's identity} and explore some of its corollaries. Finally, in Section 7, we state a one variable generalization of identity \eqref{finite sum of 2phi1} of Dixit and Patel.

\section{Main Results}
We start with a one variable generalization of the identity \eqref{bem} of Bhoria et. al. \cite[Theorem 2.1]{BEM21}.
\begin{theorem}\label{generalization of BEM's result}
Let $a,b,c,d,e$ be complex numbers such that $\big|ade\big|<1$ and $|ceq|<1$. Then
\begin{align*}
&\sum_{n=1}^\infty \frac{ (b/a)_n (c/d)_n (q/e)_{n-1}(ad)^n e^{n-1}}{(b)_n(cq)_n(q)_{n-1}} \nonumber\\ &= \frac{(a-b)(d-c)}{(ad-b)}\frac{(ceq)_\infty (ad)_\infty}{(cq)_\infty (ade)_\infty} \sum_{n=0}^\infty \frac{(a)_n (bd/c)_n (q/e)_n (ce)^n}{(b)_n(ad)_n(q)_n} \left(\frac{adq^n}{1-adq^n}-\frac{bq^n}{1-bq^n}\right).
\end{align*}
\end{theorem}
By substituting $e=1$ in the above theorem, one can easily obtain \eqref{bem}. Also, letting $a\rightarrow 0$ and then replacing $b$ by $zq$ in Theorem \ref{generalization of BEM's result}, we get the below result.
\begin{corollary}\label{4-variable generalization of Andrews Garvan Liang}
For $|ceq|<1$, we have
\begin{align}
\sum_{n=1}^\infty &\frac{ (c/d)_n (q/e)_{n-1} (-dz)^n e^{n-1} q^{n(n+1)/2}}{(zq)_n(cq)_n(q)_{n-1}} \nonumber\\ 
&= (c-d) \frac{z}{c} \frac{(ceq)_\infty}{(cq)_\infty } \sum_{n=1}^\infty \frac{(zqd/c)_{n-1}(q/e)_{n-1} (cq)^n e^{n-1}}{(zq)_n(q)_{n-1}}.
\end{align}
\end{corollary}
Letting $e=1$ in the above identity gives the following identity of Bhoria et. al. \cite[Theorem 2.2]{BEM21}:
\begin{align*}
\sum_{n=1}^\infty \frac{(-z)^n(c/d)_n d^nq^{\frac{n(n+1)}{2}}}{(cq)_n(zq)_n}=\frac{z(c-d)}{c}\sum_{n=1}^\infty \frac{(zdq/c)_{n-1}}{(zq)_n}(cq)^n.
\end{align*}
Further, setting $d \rightarrow 0$ in Corollary \ref{4-variable generalization of Andrews Garvan Liang} gives a generalization of \eqref{andrews q-series}.
\begin{corollary}\label{2-variable generalization of Andrews}
For $|ceq|<1$, we have
\begin{align}\label{2-variable generalization of Andrews equation}
\sum_{n=1}^\infty \frac{(q/e)_{n-1} e^{n-1} (cz)^nq^{n^2}}{(zq)_n (cq)_n (q)_{n-1}} = \frac{z(ceq)_\infty}{(cq)_\infty}\sum_{n=1}^\infty \frac{(q/e)_{n-1} e^{n-1} (cq)^n}{(zq)_n (q)_{n-1}} . 
\end{align}
\end{corollary}
Under the substitution $d=1$ in Theorem $\ref{generalization of BEM's result}$ gives an interesting new generalization of \eqref{dixit maji}.
\begin{corollary}\label{different generalization of dixit-maji} For $\left|ae\right|<1$ and $|ceq|<1$, we have
\begin{align}
\sum_{n=1}^\infty \frac{ (b/a)_n (q/e)_{n-1} a^n e^{n-1}}{(1-cq^n) (b)_n (q)_{n-1}} = \frac{(ceq)_\infty (a)_\infty}{(cq)_\infty (ae)_\infty}\sum_{n=0}^\infty \frac{(b/c)_n (q/e)_n (ce)^n}{(b)_n(q)_n} \left(\frac{aq^n}{1-aq^n}-\frac{bq^n}{1-bq^n}\right).
\end{align}
\end{corollary}
Here we remark that the above generalization is different from the generalization \eqref{bem} obtained by Bhoria, Eyyunni and Maji.

\subsection{Finite Analogue of Theorem \ref{generalization of BEM's result}}
We now state a finite analogue of Theorem \ref{generalization of BEM's result} which essentially generalizes the identity \eqref{dixit patel} of Dixit and Patel.
\begin{theorem}\label{finite analogue of generalization of BEM's result} For any natural number $N$, we have
\begin{align}
&\sum_{n=1}^N \begin{bmatrix}
N\\
n
\end{bmatrix} \frac{(q)_n (b/a)_n (c/d)_n (q/e)_{n-1} (ade)_{N-n} (ad)^n e^{n-1}}{(b)_n(cq)_n(q)_{n-1}} = \frac{(a-b)(d-c)}{(ad-b)} \frac{(ad)_N}{(cq)_N} \nonumber \\
& \times \sum_{n=1}^N \begin{bmatrix}
N\\
n
\end{bmatrix} \frac{(q)_n (ceq)_{N-n} (a)_{n-1} (bd/c)_{n-1} (q/e)_{n-1} (ce)^{n-1}}{(b)_{n-1}(ad)_{n-1}(q)_{n-1}} \left(\frac{adq^{n-1}}{1-adq^{n-1}}-\frac{bq^{n-1}}{1-bq^{n-1}}\right).
\end{align}
\end{theorem}
Setting $e=1$ in the above identity gives \eqref{dixit patel}. Moreover, letting $a\rightarrow 0$ and then substituting $b=zq$ in Theorem \ref{finite analogue of generalization of BEM's result}, we get a finite analogue of \eqref{4-variable generalization of Andrews Garvan Liang}.
\begin{corollary}\label{finite analogue of 4-variable generalization of Andrews Garvan Liang} We have
\begin{align}
\sum_{n=1}^N \begin{bmatrix}
N\\
n
\end{bmatrix} &\frac{(q)_n (c/d)_n (q/e)_{n-1} (-zd)^n q^{n(n+1)/2} e^{n-1} }{(zq)_n(cq)_n(q)_{n-1}} = \frac{z(c-d)}{c(cq)_N} \nonumber \\
& \times  \sum_{n=1}^N \begin{bmatrix}
N\\
n
\end{bmatrix} \frac{(q)_n(ceq)_{N-n} (zdq/c)_{n-1}(q/e)_{n-1}(cq)^n e^{n-1}}{(zq)_{n}(q)_{n-1}}.
\end{align}
\end{corollary}
Putting $e=1$ in the above identity gives the below identity of Dixit and Patel \cite[corollary 2.2]{DixitPatel},
\begin{align*}
\sum_{n=1}^N \begin{bmatrix}
N\\n
\end{bmatrix} \frac{(q)_n (c/d)_n (-zd)^n q^{n(n+1)/2}}{(zq)_n(cq)_n}= \frac{z(c-d)}{c}\sum_{n=1}^N \begin{bmatrix}
N\\n
\end{bmatrix} \frac{(q)_n(cq)_{N-n} (zdq/c)_{n-1}(cq)^n}{(zq)_{n}(cq)_N}.
\end{align*}
Further, by allowing $d\rightarrow 0$ in Corollary \ref{finite analogue of 4-variable generalization of Andrews Garvan Liang}, we arrive at the following finite analogue of the identity \eqref{2-variable generalization of Andrews equation}:
\begin{corollary}\label{finite analogue of 2-variable generalization of Andrews} We have
\begin{align*}
\sum_{n=1}^N \begin{bmatrix}
N\\
n
\end{bmatrix} \frac{(q)_n (q/e)_{n-1} (zc)^n e^{n-1} q^{n^2} }{(zq)_n(cq)_n(q)_{n-1}} = \frac{z}{(cq)_N} \sum_{n=1}^N \begin{bmatrix}
N\\
n
\end{bmatrix} \frac{(q)_n (ceq)_{N-n} (q/e)_{n-1}(cq)^n e^{n-1}}{(zq)_{n}(q)_{n-1}}.
\end{align*}
\end{corollary}
Again, substituting $d=1$ in Theorem \ref{finite analogue of generalization of BEM's result}, we get a finite analogue of \eqref{different generalization of dixit-maji}.
\begin{corollary}\label{finite analogue of different generalization of dixit-maji} We have
\begin{align*}
&\sum_{n=1}^N \begin{bmatrix}
N\\
n
\end{bmatrix} \frac{(q)_n (b/a)_n (q/e)_{n-1} (ae)_{N-n} (a)^n e^{n-1}}{(1-cq^n)(b)_n(q)_{n-1}} \nonumber \\ &= \frac{(a)_N}{(cq)_N} \sum_{n=1}^N \begin{bmatrix}
N\\
n
\end{bmatrix} \frac{(q)_n (ceq)_{N-n} (b/c)_{n-1} (q/e)_{n-1} (ce)^{n-1}}{(b)_{n-1}(q)_{n-1}} \left(\frac{aq^{n-1}}{1-aq^{n-1}}-\frac{bq^{n-1}}{1-bq^{n-1}}\right).
\end{align*}
\end{corollary}
Next, we state an identity that generalizes a finite analogue of Garvan's identity \eqref{finite analogue of garvan} due to Dixit, Eyyunnim Maji and Sood.
\begin{theorem}\label{finite analogue of generalization of garvan's identity} We have
\begin{align*}
&\sum_{n=1}^N \begin{bmatrix}
N\\
n
\end{bmatrix}_{q^2} \frac{(-1)^n(q^2;q^2)_n (z/d;q^2)_n z^nd^n q^{n^2}}{(z-d)(zq)_{2n}}\nonumber \\ &= \sum_{n=1}^N \begin{bmatrix}
N\\
n
\end{bmatrix}_{q^2} \left( \frac{(dq)_{2n-2} z^{2n-1} q^{n(2n-1)}}{(zq)_{2n-1}} + \frac{(dq)_{2n-1} z^{2n} q^{n(2n+1)}}{(zq)_{2n}} \right)\frac{(q^2;q^2)_n}{(dzq^{2N+1};q^2)_n}.
\end{align*}
\end{theorem}
Letting $N \rightarrow \infty$ in the above theorem gives a one variable generalization of an identity of Garvan \eqref{garvan}. Mainly we obtain the below identity.
\begin{theorem}
For $|dz|<1$ and $|q|<1$,
\begin{align}\label{one variable generalization to Garvan's identity equation}
\sum_{n=1}^\infty \frac{(-1)^n (z/d;q^2)_n z^nd^nq^{n^2}}{(z-d)(zq)_{2n}}=\sum_{n=1}^\infty \frac{(dq)_{n-1}z^nq^{n(n+1)/2}}{(zq)_n}.
\end{align}
\end{theorem}
Upon substituting $d=1$ in \eqref{one variable generalization to Garvan's identity equation}, one can easily get Garvan's identity \eqref{garvan}.

\section{Preliminary tools}
In this section, we gather all the necessary results that will be usefull through the paper.

For $|z|<1$, the $q-$binomial theorem is given by 
\begin{align}\label{q-Binomial theorem}
\sum_{n=0}^{\infty} \frac{(a)_nz^n}{(q)_n}=\frac{(az)_{\infty}}{(z)_{\infty}}.
\end{align}

Basic hypergeometric series  ${}_{r+1}\phi_{r}$ is define as,
\begin{align*}
{}_{2}\phi_{1}\left[ \begin{matrix} a_1, & a_2, & \cdots, &a_{r+1}  \\ b_1, &b_2, &\cdots, & b_r \end{matrix}  ; q; z \right] = \sum_{n=0}^{\infty} \frac{(a_1)_n(a_2)_n \cdots (a_{r+1})_n}{(b_1)_n(b_2)_n \cdots (b_r)_n}\frac{z^n}{(q)_n}.
\end{align*}

We recall famous Heine transformation \cite[p.~359, (III.1)]{GasperRahman}
\begin{align}\label{heine transformation}
{}_{2}\phi_{1}\left[ \begin{matrix} a, & b  \\ & c \end{matrix}  ; q; z \right]=\frac{(az)_{\infty}(b)_{\infty}}{(z)_{\infty}(c)_{\infty}}{}_{2}\phi_{1}\left[ \begin{matrix} \frac{c}{b}, & z  \\ & az \end{matrix}  ; q; b \right].
\end{align}

The $q$-Gauss sum is
\begin{align}\label{q gauss sum}
{}_{2}\phi_{1}\left[ \begin{matrix} a, & b  \\ & c \end{matrix}  ; q; \frac{c}{ab}\right] = \frac{(c/a)_\infty (c/b)_\infty}{(c)_\infty (c/ab)_\infty}.
\end{align}

We shall also recall two ${}_{3}\phi_{2}$ transformation formula in \cite[p.~359, (III.9),p.~360, (III.12)]{GasperRahman}:
\begin{align}\label{3_phi_2}
{}_{3}\phi_{2}\left[\begin{matrix} A, & B, & C \\ & D, & E \end{matrix} \, ; q, \frac{DE}{ABC}  \right] &= \frac{\big(\frac{E}{A}\big)_\infty\big( \frac{DE}{BC}\big)_{\infty }}{\big(E\big)_\infty \big(\frac{DE}{ABC}\big)_{\infty}} 
{}_{3}\phi_{2}\left[\begin{matrix}A,& \frac{D}{B} , & \frac{D}{C} \\
& D,& \frac{DE}{BC} \end{matrix} \, ; q, \frac{E}{A} \right], \\
{}_{3}\phi_{2}\left[\begin{matrix} q^{-N}, & B, & C \\ & D, & E \end{matrix} \, ; q, q  \right] &= \frac{\big(\frac{E}{C}\big)_N}{\big(E\big)_N}c^N {}_{3}\phi_{2}\left[\begin{matrix} q^{-N},& C, & \frac{D}{B} \\& D,& \frac{Cq^{1-N}}{E} \end{matrix} \, ; q, q \right]. \label{3_phi_2 finite}
\end{align}

A finite Heine transformation given by Andrews \cite[Theorem 2]{andrews07} is
\begin{align}\label{finite Heine transformation}
{}_{3}\phi_{2}\left[\begin{matrix} q^{-N}, & \alpha, & \beta \\ & \gamma, & \frac{q^{1-N}}{\tau} \end{matrix} \, ; q, q  \right] = \frac{\left(\beta \right)_N (\alpha \tau)_N}{(\gamma)_N (\tau)_N} 
{}_{3}\phi_{2}\left[\begin{matrix}q^{-N},& \frac{\gamma}{\beta} , & \tau \\
& \alpha \tau,& \frac{q^{1-N}}{\beta} \end{matrix} \, ; q, q \right].
\end{align}

We also need a corollary of \eqref{finite Heine transformation}, given below \cite[Corollary~3,~Equation~(2.7)]{andrews07}
\begin{align}\label{corollary of finite heine transformation}
{}_{3}\phi_{2}\left[\begin{matrix} q^{-N}, & \alpha, & \beta \\ & \gamma, & \frac{q^{1-N}}{\tau} \end{matrix} \, ; q, q  \right] = \frac{\left(\frac{\gamma}{\beta} \right)_N (\beta \tau)_N}{(\gamma)_N (\tau)_N} 
{}_{3}\phi_{2}\left[\begin{matrix}q^{-N},& \frac{\alpha \beta \tau}{\gamma} , & \beta \\
& \beta \tau,& \frac{\beta q^{1-N}}{\gamma} \end{matrix} \, ; q, q \right].
\end{align}

Some basic formula from \cite[Equation (4.1), (4.2)]{DEMS} stated below.
\begin{align}
\left(\frac{q^{-N}}{x}\right)_n &=\frac{(-1)^n \left( xq^{N-n+1}\right)_n}{x^nq^{Nn}}q^{\frac{n(n-1)}{2}},\label{basic formula 1} \\
\frac{(q^{-N})_n}{\left(\frac{q^{-N}}{x}\right)_n} &=\frac{(q^{N-n+1})_n x^n}{(x q^{N-n+1})_n}=\frac{(q)_N (x q)_{N-n}}{(q)_{N-n}(x q)_N} x^n. \label{basic formula 2}
\end{align}

We also state a result of Corteel and Lovejoy \cite[Corollary 1.2]{overp}
\begin{align}\label{overpartition identity}
\sum_{n=0}^N \begin{bmatrix}
N\\n
\end{bmatrix} \frac{(-1/a)_n (ac)^n q^{n(n+1)/2}}{(cq)_n} = \frac{(-acq)_N}{(cq)_N}.
\end{align}

In the next section, we present a proof of our main results.

\section{Proof of Main results}

\begin{proof}[Theorem \ref{generalization of BEM's result}][]
Let $A=\frac{q}{e},~B=\frac{bq}{a},~C=\frac{cq}{d},~D=bq, $ and $E=cq^2$ in \eqref{3_phi_2} so that
\begin{align*}
\sum_{n=0}^\infty \frac{(q/e)_n (bq/a)_n (cq/d)_n (ade)^n }{(bq)_n (cq^2)_n (q)_n} = \frac{(ceq)_\infty (adq)_\infty }{ (cq^2)_\infty (ade)_\infty } \sum_{n=0}^\infty \frac{ (q/e)_n (a)_n (bd/c)_n (ceq)^n }{ (bq)_n (adq)_n (q)_n }.
\end{align*}
Re-indexing the summation on left hand side and then multiplying both sides by $\frac{(1-b/a)(1-c/d) ad }{(1-b)(1-cq)}$, we get
\begin{align*}
&\sum_{n=1}^\infty \frac{ (b/a)_n (c/d)_n (q/e)_{n-1} (ad)^n e^{n-1} }{ (b)_n (cq)_n (q)_{n-1}}\\
&= \frac{(a-b)(d-c) }{(1-b)(1-ad)} \frac{(ceq)_\infty (ad)_\infty }{ (cq)_\infty (ade)_\infty }  \sum_{n=0}^\infty \frac{ (q/e)_n (a)_n (bd/c)_n (ceq)^n }{ (bq)_{n} (adq)_n (q)_n }\\
&= \frac{(a-b)(d-c) }{(ad-b)} \frac{(ceq)_\infty (ad)_\infty }{ (cq)_\infty (ade)_\infty }  \sum_{n=0}^\infty \frac{ (a)_n (bd/c)_n (q/e)_n (ce)^n }{ (b)_{n} (ad)_n (q)_n } \left(\frac{(ad-b)q^n}{(1-bq^n)(1-adq^n)}\right)\\
&= \frac{(a-b)(d-c) }{(ad-b)} \frac{(ceq)_\infty (ad)_\infty }{ (cq)_\infty (ade)_\infty }  \sum_{n=0}^\infty \frac{ (a)_n (bd/c)_n (q/e)_n (ce)^n }{ (b)_{n} (ad)_n (q)_n } \left(\frac{adq^n}{1-adq^n}-\frac{bq^n}{1-bq^n}\right).
\end{align*}
This proves the result.
\end{proof}

\begin{proof}[Theorem \ref{finite analogue of generalization of BEM's result}][]
We know that from \cite[p.~70, Equation (3.2.1)]{GasperRahman}
\begin{align*}
{}_{4}\phi_{3}\left[\begin{matrix} q^{-N}, &  A, & B, & C \\ & D, & E, & \frac{ABCq^{1-N}}{DE} \end{matrix} \, ; q, q  \right] = \frac{\big(\frac{E}{A}\big)_N\big( \frac{DE}{BC}\big)_{N}}{\big(E\big)_N \big(\frac{DE}{ABC}\big)_{N}} 
{}_{4}\phi_{3}\left[\begin{matrix} q^{-N}, & A, & \frac{D}{B} , & \frac{D}{C} \\
& D,& \frac{DE}{BC}, & \frac{Aq^{1-N}}{E} \end{matrix} \, ; q, q \right].
\end{align*}
Substituting $A=\frac{q}{e},~B=\frac{bq}{a},~C=\frac{cq}{d},~D=bq, $ and $E=cq^2$ in the above identity, we get
\begin{align*}
\sum_{n=0}^N \frac{\left(q^{-N}\right)_n (q/e)_n (bq/a)_n (cq/d)_n }{ (bq)_n (cq^2)_n (q^{1-N}/ade)_n (q)_n }q^n = \frac{(ceq)_N (adq)_N }{ (cq^2)_N (ade)_N } \sum_{n=0}^N \frac{ \left(q^{-N}\right)_n (q/e)_n (a)_n (bd/c)_n }{ (bq)_n (adq)_n (q^{-N}/ce)_n (q)_n } q^n.
\end{align*}
Now we use \eqref{basic formula 2} with $x=ade/q$ for left hand side and with $x=ce$ for the right hand side of the aforementioned expression to see that
\begin{align*}
\sum_{n=0}^N &\frac{(q)_N (ade)_{N-n} (q/e)_n (bq/a)_n (cq/d)_n (ade)^n }{ (q)_{N-n}  (bq)_n (cq^2)_n (q)_n }= \frac{ (adq)_N }{ (cq^2)_N } \sum_{n=0}^N \frac{ (q)_N (ceq)_{N-n} (q/e)_n (a)_n (bd/c)_n (ceq)^n }{ (q)_{N-n} (bq)_n (adq)_n (q)_n }.
\end{align*}
Re-indexing the sums of both sides by replacing $n$ by $n-1$ and then multiplying the resultant by $\frac{(1-q^N)(1-b/a)(1-c/d)ad}{(1-b)(1-cq)}$ yields
\begin{align*}
&\sum_{n=1}^{N+1} \begin{bmatrix}
N+1\\
n
\end{bmatrix}
\frac{(q)_n(ade)_{N-n+1} (q/e)_{n-1} (b/a)_n (c/d)_n (ad)^n e^{n-1} }{ (b)_n (cq)_n (q)_{n-1} }\\
&= (a-b)(d-c)\frac{(adq)_N }{ (cq)_{N+1}} \sum_{n=1}^{N+1} \begin{bmatrix}
N+1\\
n
\end{bmatrix} \frac{(q)_{n} (ceq)_{N-n+1} (q/e)_{n-1} (a)_{n-1} (bd/c)_{n-1} (ceq)^{n-1} }{ (b)_n (adq)_{n-1} (q)_{n-1}}\\
&=\frac{(a-b)(d-c)}{(ad-b)} \frac{ (ad)_{N+1} }{ (cq)_{N+1}} \sum_{n=1}^{N+1} \begin{bmatrix}
N+1\\
n
\end{bmatrix} \frac{(q)_{n} (a)_{n-1} (bd/c)_{n-1} (ceq)_{N-n+1} (q/e)_{n-1} (ce)^{n-1} }{ (b)_{n-1} (ad)_{n-1} (q)_{n-1}}\\
&\hspace{7cm}\times \left(\frac{adq^{n-1}}{1-adq^{n-1}} - \frac{bq^{n-1}}{1-bq^{n-1}} \right).
\end{align*}
Now replace $N$ by $N-1$ to get the desired result.
\end{proof}

\section{One variable generalization of five entries of Ramanujan and their finite analogue}
In this section, our aim is to derive a one variable generalization of each of five entries of Ramanujan and their corresponding finite analogues. We will start with one variable generalization of Entry 1 \eqref{entry 1}.

\begin{theorem}[One variable generalization of Entry 1]\label{one variable generalization of entry 1}
For $|ae|<1$, we have
\begin{align}\label{one variable generalization of entry 1 equation}
\frac{(-a q)_\infty (be)_\infty}{(-ae)_\infty  (b q)_\infty}=\sum_{n=0}^{\infty} \frac{ (-1)^n (-b/a)_n (q/e)_n (ae)^{n}}{(q)_{n}(b q)_{n}}.
\end{align}
\end{theorem}

\begin{proof}
Letting $a$ tends to $0$ and then substituting $d=1/q$ and $c=-a/q$ in Theorem \ref{generalization of BEM's result}, we get

\begin{align*}
\sum_{n=1}^\infty \frac{ (-1)^{n-1} (q/e)_{n-1} b^n e^{n-1} q^{n(n-3)/2}}{(b)_n (q)_{n-1}} = \frac{b(1+a)}{q(1-b)}  \frac{(-ae)_\infty }{(-a)_\infty } \sum_{n=0}^\infty \frac{ (-1)^{n} (-b/a)_n (q/e)_n (ae)^n}{(bq)_{n} (q)_n}.
\end{align*}
Multiplying both sides by $\frac{q(1-b) (-a)_\infty}{b(1+a)(-ae)_\infty}$ and then re-indexing the left side sum of the resulting identity gives
\begin{align}\label{1-calculation for one variable gen of entry 1}
\frac{(-aq)_\infty}{(-ae)_\infty} \sum_{n=0}^\infty \frac{ (-1)^{n} (q/e)_{n} b^{n} e^{n} q^{n(n-1)/2}}{(bq)_{n} (q)_{n}} = \sum_{n=0}^\infty \frac{ (-1)^{n} (-b/a)_n (q/e)_n (ae)^n}{(bq)_{n} (q)_n}.
\end{align}
The left hand side sum of the above expression can be written as
\begin{align*}
\sum_{n=0}^\infty \frac{ (-1)^{n} (q/e)_{n} b^{n} e^{n} q^{n(n-1)/2}}{(bq)_{n} (q)_{n}} = \lim_{\alpha \rightarrow \infty} {}_{2}\phi_{1}\left[ \begin{matrix} q/e, & b/\alpha  \\ & bq \end{matrix}  ; q; \alpha e \right].
\end{align*}
Apply $q$-Gauss sum \eqref{q gauss sum} in the right hand side of the above identity to see that 
\begin{align}\label{2-calculation for one variable gen of entry 1}
\sum_{n=0}^\infty \frac{ (-1)^{n} (q/e)_{n} b^{n} e^{n} q^{n(n-1)/2}}{(bq)_{n} (q)_{n}} = \frac{(be)_\infty}{(bq)_\infty}.
\end{align}
Employing \eqref{2-calculation for one variable gen of entry 1} in \eqref{1-calculation for one variable gen of entry 1}, the proof of \eqref{one variable generalization of entry 1 equation} follows.
\end{proof}
\begin{remark}
It can be easily observed that letting $e \rightarrow 0$ in Theorem \ref{one variable generalization of entry 1} gives Ramanujan's identity i.e.,  Entry 1 \eqref{entry 1}.
\end{remark}

A finite analogue of Theorem \ref{one variable generalization of entry 1} is given below.

\begin{theorem}[Finite analogue of Entry 1]\label{finite analogue of one variable generalization of entry 1}
For $N\in \mathbb{N}$, we have
\begin{align}\label{finite analogue of one variable generalization of entry 1 equation}
\frac{(-a q)_N (be)_N}{(b q)_N}=\sum_{n=0}^{N} \begin{bmatrix}
N\\n
\end{bmatrix} \frac{ (-1)^n (-ae)_{N-n} (-b/a)_n (q/e)_n (ae)^{n}}{(b q)_{n}}.
\end{align}
\end{theorem}

\begin{proof}
Letting $a \rightarrow 0$, then substituting $d = 1/q$ and $c = -a/q$ in \eqref{finite analogue of generalization of BEM's result}, we get
\begin{align*}
\sum_{n=1}^N  \frac{(q)_N (q/e)_{n-1} (-be)^{n-1} q^{(n-1)(n-2)/2} }{(q)_{N-n} (b)_n (q)_{n-1}} = \sum_{n=1}^N \frac{(q)_N (-ae)_{N-n} (-b/a)_{n-1} (q/e)_{n-1} (-ae)^{n-1}}{ (q)_{N-n} (-aq)_{N-1} (b)_{n} (q)_{n-1}}.
\end{align*}
Re-indexing both sums above and then replacing $N$ by $N+1$ yields
\begin{align*}
\sum_{n=0}^N  \frac{(q)_{N+1} (q/e)_{n} (-be)^{n} q^{n(n-1)/2} }{(q)_{N-n} (b)_{n+1} (q)_{n}} = \sum_{n=0}^N \frac{(q)_{N+1} (-ae)_{N-n} (-b/a)_{n} (q/e)_{n} (-ae)^{n}}{ (q)_{N-n} (-aq)_{N} (b)_{n+1} (q)_{n}}.
\end{align*}
Multiply both sides by $\frac{1-b}{1-q^{N+1}}$ in the above equation and then further simplifying, one gets
\begin{align*}
(-aq)_{N} \sum_{n=0}^N \begin{bmatrix}
N\\n
\end{bmatrix}  \frac{ (q/e)_{n} (-be)^{n} q^{n(n-1)/2} }{ (bq)_{n} } = \sum_{n=0}^N \begin{bmatrix}
N\\n
\end{bmatrix} \frac{ (-ae)_{N-n} (-b/a)_{n} (q/e)_{n} (-ae)^{n}}{ (bq)_{n} }.
\end{align*}
Now the result follows by using \eqref{overpartition identity} with $a=-e/q,~c=b$ in the above identity.
\end{proof}
Allowing $N \rightarrow \infty$ in Theorem \ref{finite analogue of one variable generalization of entry 1} gives Theorem \ref{one variable generalization of entry 1}. Also, letting $e \rightarrow 0$ in \eqref{finite analogue of one variable generalization of entry 1 equation} gives
\begin{align}\label{1st step to dixit patel's finite analogue of entry 1}
\frac{(-a q)_N}{(b q)_N}=\sum_{n=0}^{N} \begin{bmatrix}
N\\n
\end{bmatrix} \frac{ (-b/a)_n a^n q^{n(n+1)/2}}{(b q)_{n}}.
\end{align}
By doing a simple calculation, one can observe that
\begin{align}\label{2nd step to dixit patel's finite analogue of entry 1}
\frac{(-a q)_N}{(b q)_N}= \frac{\left(q^{-N}/a\right)_N}{\left(q^{-N}/b\right)_N} \left(\frac{-a}{b} \right)^N.
\end{align}
Apply $q$-Chu-Vandermonde sum \cite[p.~354,~(II.6)]{GasperRahman} in the right side of the above expression to get
\begin{align*}
\frac{\left(q^{-N}/a\right)_N}{\left(q^{-N}/b\right)_N} \left(\frac{-a}{b} \right)^N &= {}_{2}\phi_{1}\left[ \begin{matrix} -a/b, & q^{-N}  \\ & q^{-N}/b \end{matrix}  ; q; z \right] \\
&=\sum_{n=0}^N \frac{(q^{-N})_n (-a/b)_n}{(q^{-N}/b)_n (q)_n}q^n.
\end{align*}
Apply \eqref{basic formula 2} with $x=b$ in the above expression to obtain
\begin{align}\label{3rd step to dixit patel's finite analogue of entry 1}
\frac{\left(q^{-N}/a\right)_N}{\left(q^{-N}/b\right)_N} \left(\frac{-a}{b} \right)^N &= \sum_{n=0}^N \frac{(q)_{N} (bq)_{N-n} (-a/b)_n}{(q)_{N-n} (bq)_N (q)_n} (bq)^n \nonumber \\
&=\sum_{n=0}^N \begin{bmatrix}
N\\n
\end{bmatrix} \frac{ (-a/b)_n (bq)_{N-n} (bq)^n}{(bq)_N}.
\end{align}
Upon employing \eqref{2nd step to dixit patel's finite analogue of entry 1} and \eqref{3rd step to dixit patel's finite analogue of entry 1} in \eqref{1st step to dixit patel's finite analogue of entry 1} gives a finite analogue of Entry 1 due to Dixit and Patel \cite[Theorem 3.1]{DixitPatel}.

The following identity is one variable generalization of Entry 2 \eqref{entry 2}.

\begin{theorem}[One variable generalization of Entry 2]\label{one variable generalization of entry 2}
For $|ae|<1$, we have
\begin{align}\label{one variable generalization of entry 2 equation}
\frac{(a q)_\infty }{(ae)_\infty } \sum_{n=1}^\infty \frac{(-1)^n n (q/e)_n (ae)^n q^{n(n-1)/2}}{(aq)_n (q)_n}=-\sum_{n=1}^{\infty} \frac{(q/e)_n (ae)^n}{1-q^n}.
\end{align}
\end{theorem}

\begin{proof}
Letting $d \rightarrow 0$ in Theorem \ref{generalization of BEM's result}, we get
\begin{align*}
\sum_{n=1}^\infty \frac{(-1)^{n-1} (b/a)_n (q/e)_{n-1}(ac)^n e^{n-1}q^{n(n-1)/2}}{(b)_n(cq)_n(q)_{n-1}} = c(a-b)\frac{(ceq)_\infty }{(cq)_\infty } \sum_{n=0}^\infty \frac{(a)_n (q/e)_n (ceq)^n}{(b)_{n+1} (q)_n}.
\end{align*}
Next, putting $b=0$ and then re-indexing the left hand side sum we have
\begin{align*}
\sum_{n=0}^\infty \frac{(-1)^{n} (q/e)_{n}(ac)^{n+1} e^{n}q^{n(n+1)/2}}{(cq)_{n+1} (q)_{n}} = ac \frac{(ceq)_\infty }{(cq)_\infty } \sum_{n=0}^\infty \frac{(a)_n (q/e)_n (ceq)^n}{ (q)_n}.
\end{align*}
Multiply both sides by $\frac{(cq)_\infty}{ac(ceq)_\infty}$ and then differentiating with respect to $a$ yields
\begin{align*}
\frac{(cq^2)_\infty}{(ceq)_\infty} \sum_{n=1}^\infty \frac{(-1)^{n} n (q/e)_{n} a^{n-1} (ce)^nq^{n(n+1)/2}}{(cq^2)_{n} (q)_{n}} = \sum_{n=1}^\infty \frac{(a)_n (q/e)_n (ceq)^n}{ (q)_n}\sum_{k=0}^{n-1}\frac{-q^k}{1-aq^k} .
\end{align*}
Now letting $a$ tends to $1$ in the above expression and using the fact
\begin{align*}
\lim _{a \rightarrow 1}\left(a\right)_{n} \sum_{k=0}^{n-1}\left(\frac{q^{k}}{1-aq^{k}}\right) & =\lim _{a \rightarrow 1}\left[\frac{(a)_{n}}{1-a}+\left(a\right)_{n}\left(\frac{q}{1-aq}+\cdots+\frac{q^{n-1}}{1- aq^{n-1}}\right)\right] =(q)_{n-1},
\end{align*} in the right hand side expression, we get
\begin{align*}
\frac{(cq^2)_\infty}{(ceq)_\infty} \sum_{n=1}^\infty \frac{(-1)^{n} n (q/e)_{n} (ce)^n q^{n(n+1)/2}}{(cq^2)_{n} (q)_{n}} = -\sum_{n=1}^\infty \frac{ (q/e)_n (ceq)^n}{ 1-q^n}.
\end{align*}
The result \eqref{one variable generalization of entry 2 equation} follows by substituting $c=a/q$ in the above expression.
\end{proof}

One can easily check that, Ramanujan's Entry 2 \eqref{entry 2} follows from Theorem \ref{one variable generalization of entry 2} by letting $e \rightarrow 0$. A finite analogue of Theorem \ref{one variable generalization of entry 2} is stated below.

\begin{theorem}[Finite analogue of Entry 2]\label{finite analogue of one variable generalization of entry 2}
For $N\in \mathbb{N}$, we have
\begin{align}\label{finite analogue of one variable generalization of entry 2 equation}
(a q)_N \sum_{n=1}^N \begin{bmatrix}
N\\n
\end{bmatrix} \frac{(-1)^{n-1} n (q/e)_n (ae)^n q^{n(n-1)/2}}{(aq)_n}=\sum_{n=1}^{N}\begin{bmatrix}
N\\n
\end{bmatrix} \frac{(ae)_{N-n} (q)_n (q/e)_n (ae)^n}{1-q^n}.
\end{align}
\end{theorem}

\begin{proof}
Letting $d \rightarrow 0$ in Theorem \ref{finite analogue of generalization of BEM's result}, we get
\begin{align*}
\sum_{n=1}^N \frac{(q)_N (b/a)_n (q/e)_{n-1} (-ace)^{n-1} q^{n(n-1)/2}}{ (q)_{N-n} (b)_n(cq)_n(q)_{n-1}}=\frac{(a-b)}{a(cq)_N} \sum_{n=1}^N \frac{(q)_N (ceq)_{N-n} (a)_{n-1} (q/e)_{n-1} (ceq)^{n-1}}{ (q)_{N-n} (b)_{n} (q)_{n-1}}.
\end{align*}
Now substitute $b=0$ in the above identity and then re-index the sum of both sides to obtain
\begin{align*}
\sum_{n=0}^{N-1} \frac{(q)_N (q/e)_{n} (-ace)^{n} q^{n(n+1)/2}}{ (q)_{N-n-1} (cq)_{n+1} (q)_{n}} = \frac{1}{(cq)_N} \sum_{n=0}^{N-1} \frac{(q)_N (ceq)_{N-n-1} (a)_{n} (q/e)_{n} (ceq)^{n}}{ (q)_{N-n-1} (q)_{n}}.
\end{align*}
Replacing $N$ by $N+1$ and then multiplying both sides by $\frac{(cq)_{N+1}}{1-q^{N+1}}$ in the resulting expression yields
\begin{align*}
(cq^2)_N \sum_{n=0}^{N}\begin{bmatrix}
N\\n
\end{bmatrix} \frac{(-1)^n (q/e)_{n} (ace)^{n} q^{n(n+1)/2}}{ (cq^2)_{n} } =  \sum_{n=0}^{N}\begin{bmatrix}
N\\n
\end{bmatrix} (ceq)_{N-n} (a)_{n} (q/e)_{n} (ceq)^{n}.
\end{align*}
Now differentiate this identity with respect to the variable $a$ to get
\begin{align*}
(cq^2)_N \sum_{n=0}^{N}\begin{bmatrix}
N\\n
\end{bmatrix} \frac{(-1)^n (q/e)_{n}na^{n-1} (ce)^{n} q^{n(n+1)/2}}{ (cq^2)_{n} } =  \sum_{n=0}^{N}\begin{bmatrix}
N\\n
\end{bmatrix} (ceq)_{N-n} (a)_{n} (q/e)_{n} (ceq)^{n}\sum_{k=0}^{n-1}\frac{-q^k}{1-aq^k}.
\end{align*}
Finally, let $a \rightarrow 1$ and then substitute $c=a/q$ in the above identity and use the fact that
\begin{align*}
\lim_{a \rightarrow 1} (a)_n \sum_{k=0}^{n-1} \frac{q^k}{1-aq^k} = \lim_{a\rightarrow 1} \left[\frac{(a)_n}{1-a} + (a)_n \sum_{k=1}^{n-1} \frac{q^k}{1-aq^k} \right]=(q)_{n-1},
\end{align*}
to prove \eqref{finite analogue of one variable generalization of entry 2 equation}.
\end{proof}

\begin{remark}
Letting $N\rightarrow \infty$ in Theorem \ref{finite analogue of one variable generalization of entry 2} allows us to obtain Theorem \ref{one variable generalization of entry 2}. Also, $e \rightarrow 0$ leads us to an identity of Dixit and Patel \cite[Theorem 3.2]{DixitPatel} which essentially serves as a finite analogue of Ramanujan's Entry 2 \eqref{entry 2},
\begin{align*}
(a q)_N \sum_{n=1}^N \begin{bmatrix}
N\\n
\end{bmatrix} \frac{n a^n q^{n^2}}{(aq)_n}=\sum_{n=1}^{N}\begin{bmatrix}
N\\n
\end{bmatrix} \frac{(q)_n (-1)^{n-1} a^n q^{n(n+1)/2}}{1-q^n}.
\end{align*}
\end{remark}
The next result is a one variable generalization of Entry 3 \eqref{entry 3}.

\begin{theorem}[One variable generalization of Entry 3]\label{one variable generalization of entry 3}
For $|b|<1$, we have
\begin{equation}\label{one variable generalization of entry 3 equation}
\sum_{n=1}^\infty \frac{(-1)^n(b/a)_n a^n q^{n(n+1)/2}}{(1-q^n) (b)_n (q/e)_n e^n} = \sum_{n=1}^\infty \frac{(aq/be)_n b^n}{(1-q^n) (q/e)_n} - \sum_{n=1}^\infty \frac{b^n}{1-q^n}.
\end{equation}
\end{theorem}

\begin{proof}
Differentiating \eqref{4-variable generalization of Andrews Garvan Liang} with respect to $e$ gives
\begin{align*}
& \sum_{n=2}^\infty \frac{(-1)^n (c/d)_n (dz)^n q^{n(n+1)/2}}{(zq)_n(cq)_n(q)_{n-1}} \frac{\mathrm{d}}{\mathrm{d}e} \left(e^{n-1} (q/e)_{n-1} \right) =  \frac{zq(c-d)}{(1-zq) (cq)_\infty } \frac{\mathrm{d}}{\mathrm{d}e} (ceq)_\infty\\
&\hspace{30mm} + \frac{z(c-d)}{c(cq)_\infty } \sum_{n=2}^\infty \frac{(zqd/c)_{n-1} (cq)^n}{(zq)_n(q)_{n-1}} \frac{\mathrm{d}}{\mathrm{d}e} \left(e^{n-1} (q/e)_{n-1} (ceq)_\infty \right).
\end{align*}
Letting $e$ tends to $q$ in the above expression and simplifying with the fact that
\begin{align}\label{differentiation with respect to e}
\lim_{e\rightarrow q} \left[ \frac{\mathrm{d}}{\mathrm{d}e} \left(e^{n-1} (q/e)_{n-1} \right) \right] = q^{n-2} (q)_{n-2},
\end{align}
and
\begin{align*}
\lim_{e \rightarrow q} \left[ \frac{\mathrm{d}}{\mathrm{d}e} (ceq)_\infty \right] = (cq^2)_\infty \sum_{n=1}^\infty \frac{-cq^n}{1-cq^{n+1}},
\end{align*} we get
\begin{align*}
\sum_{n=2}^\infty &\frac{(-1)^n (c/d)_n (dz)^n q^{(n-1)(n+4)/2}}{(zq)_n(cq)_n(1-q^{n-1})} =  \frac{-z(c-d)}{(1-zq) (1-cq)} \sum_{n=1}^\infty \frac{cq^{n+1}}{1-cq^{n+1}} \\
& \hspace{30mm} + \frac{z(c-d)}{c(1-cq) } \sum_{n=2}^\infty \frac{(zqd/c)_{n-1} c^n q^{2n-2}}{(zq)_n(1-q^{n-1})}.
\end{align*}
Re-index the summation and then multiply both sides by $\frac{(1-zq)(1-cq)}{z(c-d)}$ to get
\begin{align*}
\sum_{n=1}^\infty \frac{(-1)^n (cq/d)_n (dz)^n q^{n(n+5)/2}}{ (1-q^n)(zq^2)_n(cq^2)_n} =  -\sum_{n=0}^\infty \frac{cq^{n+2}}{1-cq^{n+2}}  + \sum_{n=1}^\infty \frac{(zqd/c)_{n} c^n q^{2n}}{(1-q^n) (zq^2)_n}.
\end{align*}
Now substitute $c=b/q^2,~d=a/q$ and $z=1/eq$ in the above expression and simplify to get \eqref{one variable generalization of entry 3 equation}.
\end{proof}

\begin{remark}
Letting $e \rightarrow 0$ in \eqref{one variable generalization of entry 3 equation} and using the fact that
\begin{align*}
\lim_{e \rightarrow 0} \frac{(aq/be)_n}{(q/e)_n}= \left(\frac{a}{b}\right)^n
\end{align*}
serves Entry 3 \eqref{entry 3}. We can also recover Entry 1 \eqref{entry 1} from \eqref{one variable generalization of entry 3 equation}. Multiply both sides of \eqref{one variable generalization of entry 3 equation} by $1-\frac{q}{e}$ and then take $e \rightarrow q$ in the resultant identity gives
\begin{align*}
\sum_{n=1}^\infty \frac{(-1)^n(b/a)_n a^n q^{n(n-1)/2}}{ (b)_n (q)_n } = \sum_{n=1}^\infty \frac{(a/b)_n b^n}{(q)_n}.
\end{align*}
Now use $q$-binomial theorem \eqref{q-Binomial theorem} and then replace $a$ by $-aq$ and $b$ by $bq$ to get \eqref{entry 1}. 
\end{remark}

A finite analogue of Theorem \ref{one variable generalization of entry 3} is given below.

\begin{theorem}[Finite analogue of Entry 3]\label{finite analogue of one variable generalization of entry 3}
For a natural number $N$, we have
\begin{align}\label{finite analogue of one variable generalization of entry 3 equation}
\sum_{n=1}^{N} \begin{bmatrix}
N\\
n
\end{bmatrix} \frac{(-1)^n (q)_{n-1} (b/a)_{n} a^{n} q^{n(n+1)/2} }{(b)_{n} (q/e)_{n} e^n} = \sum_{n=1}^N \begin{bmatrix}
N\\
n
\end{bmatrix} \frac{(q)_{n-1} (b)_{N-n} (aq/be)_{n} b^n}{(b)_N (q/e)_{n}} - \sum_{n=1}^{N} \frac{bq^{n-1} }{1-bq^{n-1}}.
\end{align}
\end{theorem}

\begin{proof}
Differentiating both sides of \eqref{finite analogue of 4-variable generalization of Andrews Garvan Liang} with respect to $e$ yields
\begin{align*}
\sum_{n=2}^N \begin{bmatrix}
N\\
n
\end{bmatrix} &\frac{(q)_n (c/d)_n (-zd)^n q^{n(n+1)/2} }{(zq)_n(cq)_n(q)_{n-1}} \frac{\mathrm{d}}{\mathrm{d}e} \Big( e^{n-1} (q/e)_{n-1} \Big)= \frac{zq(1-q^N)(c-d)}{(1-zq)(cq)_N} \frac{\mathrm{d}}{\mathrm{d}e}(ceq)_{N-1} \\
 &+\frac{zq(c-d)}{(cq)_N} \sum_{n=2}^N \begin{bmatrix}
N\\
n
\end{bmatrix} \frac{(q)_n (zdq/c)_{n-1} (cq)^{n-1}}{(zq)_{n}(q)_{n-1}} \frac{\mathrm{d}}{\mathrm{d}e} \Big( e^{n-1} (q/e)_{n-1} (ceq)_{N-n} \Big) .
\end{align*}
Letting $e \rightarrow q$ in the above expression, together with \eqref{differentiation with respect to e} and
\begin{align*}
\lim_{e \rightarrow q} \frac{d}{de}(ceq)_{N-1} = (cq^2)_{N-1} \sum_{n=1}^{N-1} \frac{-cq^n}{1-cq^{n+1}},
\end{align*}
gives
\begin{align*}
\sum_{n=2}^N \begin{bmatrix}
N\\
n
\end{bmatrix} &\frac{(-1)^n(q)_n (c/d)_n (zd)^n q^{(n-1)(n+4)/2} }{(zq)_n(cq)_n(1-q^{n-1})} = \frac{z(1-q^N)(c-d)}{(1-zq)(1-cq)} \sum_{k=1}^{N-1} \frac{-cq^{k+1}}{1-cq^{k+1}} \\
&\hspace{20mm}+\frac{z(c-d)}{(cq)_N} \sum_{n=2}^N \begin{bmatrix}
N\\
n
\end{bmatrix} \frac{(q)_n (cq^2)_{N-n} (zdq/c)_{n-1} c^{n-1}q^{2n-2}}{(zq)_{n}(1-q^{n-1}) }.
\end{align*}
Replace $N$ by $N+1$ and then re-index the first and the last sum to get
\begin{align*}
\sum_{n=1}^{N} &\frac{(-1)^{n+1}(q)_{N+1} (c/d)_{n+1} (zd)^{n+1} q^{n(n+5)/2} }{(q)_{N-n} (zq)_{n+1} (cq)_{n+1} (1-q^{n})} = \frac{z(1-q^{N+1})(c-d)}{(1-zq)(1-cq)} \sum_{n=1}^{N} \frac{-cq^{n+1} }{1-cq^{n+1}} \\
&\hspace{50mm}+\frac{z(c-d)}{(cq)_{N+1}} \sum_{n=1}^N \frac{(q)_{N+1} (cq^2)_{N-n} (zdq/c)_{n} c^n q^{2n}}{(q)_{N-n} (zq)_{n+1}(1-q^{n}) }.
\end{align*}
Multiplying both sides of the above identity by $\frac{(1-zq)(1-cq)}{z(1-q^{N+1})(c-d)}$, we arrive
\begin{align*}
\sum_{n=1}^{N} \begin{bmatrix}
N\\
n
\end{bmatrix} \frac{(-1)^n (q)_{n-1} (cq/d)_{n} (zd)^{n} q^{n(n+5)/2} }{(zq^2)_{n} (cq^2)_{n}} &= \sum_{n=1}^N \begin{bmatrix}
N\\
n
\end{bmatrix} \frac{(q)_{n-1} (cq^2)_{N-n} (zdq/c)_{n} c^n q^{2n}}{(cq^2)_N (zq^2)_{n}} \\
& \hspace{2cm} - \sum_{n=1}^{N} \frac{cq^{n+1} }{1-cq^{n+1}}.
\end{align*}
Now replace $c$ by $b/q^2,~ d$ by $a/q$ and  $z$ by $1/eq$ to get
\begin{align*}
\sum_{n=1}^{N} \begin{bmatrix}
N\\
n
\end{bmatrix} \frac{(-1)^n(q)_{n-1} (b/a)_{n} a^{n} q^{n(n+1)/2} }{(b)_{n} (q/e)_{n} e^n} = \sum_{n=1}^N \begin{bmatrix}
N\\
n
\end{bmatrix} \frac{(q)_{n-1} (b)_{N-n} (aq/be)_{n} b^n}{(b)_N (q/e)_{n}} - \sum_{n=1}^{N} \frac{bq^{n-1} }{1-bq^{n-1}}.
\end{align*}
This proves \eqref{finite analogue of one variable generalization of entry 3 equation}.
\end{proof}

By letting $N\rightarrow \infty$ in \eqref{finite analogue of one variable generalization of entry 3 equation}, one can easily get \eqref{one variable generalization of entry 3 equation} and $e \rightarrow 0$ gives finite analogue of \eqref{entry 3} which is different from the identity of Dixit and Patel \cite[Theorem 3.3]{DixitPatel}.

Allowing $a \rightarrow 0$ and then replace $b$ by $zq$ in Theorem \ref{one variable generalization of entry 3} gives a one variable generalization of Entry 4 \eqref{entry 4}, namely
\begin{theorem}[One variable generalization of Entry 4]\label{one variable generalization of entry 4}
For $|zq|<1$, we have
\begin{equation}\label{one variable generalization of entry 4 equation}
\sum_{n=1}^\infty \frac{z^n q^{n(n+1)}}{(1-q^n) (zq)_n (q/e)_n e^n} = \sum_{n=1}^\infty \frac{(zq)^n}{(1-q^n)} \left(\frac{1}{(q/e)_n} - 1\right).
\end{equation}
\end{theorem}
Similarly, by letting $a \rightarrow 0$ and replacing $b$ by $zq$ in Theorem \ref{finite analogue of one variable generalization of entry 3}, we get a finite analogue of the above theorem.

\begin{theorem}[Finite analogue of Entry 4]\label{finite analogue of one variable generalization of entry 4}
Let $N\in \mathbb{N}$. Then we have
\begin{align}\label{finite analogue of one variable generalization of entry 4 equation}
\sum_{n=1}^{N} \begin{bmatrix}
N\\
n
\end{bmatrix} \frac{(q)_{n-1} z^{n} q^{n(n+1)} }{(zq)_{n} (q/e)_{n} e^n} = \sum_{n=1}^N \begin{bmatrix}
N\\
n
\end{bmatrix} \frac{(q)_{n-1} (zq)_{N-n} (aq)^n}{(zq)_N (q/e)_{n}} - \sum_{n=1}^N \frac{zq^{n} }{1-zq^{n}}.
\end{align}
\end{theorem}
Clearly, it can be observed that letting $N\rightarrow \infty$ in \eqref{finite analogue of one variable generalization of entry 4 equation} gives \eqref{one variable generalization of entry 4 equation}.  Also, $e \rightarrow 0$ allow us to recover a finite analogue of Entry 4 \eqref{entry 4} given by Dixit and Patel \cite[Theorem 3.4]{DixitPatel},
\begin{align*}
\sum_{n=1}^{N} \begin{bmatrix}
N\\
n
\end{bmatrix} \frac{(q)_{n-1} (-1)^{n-1} z^{n} q^{n(n+1)/2} }{(zq)_{n}} = \sum_{n=1}^N \frac{zq^{n} }{1-zq^{n}}.
\end{align*}

A one variable generalization of Entry 5 \eqref{entry 5} is stated below.

\begin{theorem}[One variable generalization of Entry 5]\label{one variable generalization of entry 5}
For $|a|<1$, one can see
\begin{align}\label{one variable generalization of entry 5 equation}
\sum_{n=1}^\infty \frac{(-1)^n(q)_{n-1}a^n q^{n(n+1)/2}}{(1-q^n) (a)_n (q/e)_n e^n} = -\sum_{n=1}^\infty \frac{a^n}{1-q^n} \sum_{k=1}^n \frac{q^k}{e-q^k}.
\end{align}
\end{theorem}

\begin{proof}
Divide both sides of \eqref{one variable generalization of entry 3 equation} by $1-b/a$ and then let $b \rightarrow a$ and simplify to get \eqref{one variable generalization of entry 5 equation}.
\end{proof}
The next identity is a finite analogue of Theorem \ref{one variable generalization of entry 5}.
\begin{theorem}[Finite analogue of Entry 5]\label{finite analogue of one variable generalization of entry 5}
For a natural number $N$, the following identity is true:
\begin{align*}
\sum_{n=1}^{N} \begin{bmatrix}
N\\
n
\end{bmatrix} \frac{(-1)^n (q)_{n-1}^2 a^{n} q^{n(n+1)/2} }{(a)_{n} (q/e)_{n} e^n} = -\sum_{n=1}^N \begin{bmatrix}
N\\
n
\end{bmatrix} \frac{(q)_{n-1} (a)_{N-n} a^n}{(a)_N } \left(\sum_{k=1}^n \frac{q^k}{e-q^{k}}\right).
\end{align*}
\end{theorem}

\begin{proof}
Utilizing \eqref{finite analogue of dixit maji} with $b=0$ and $c=1$ in Theorem \ref{finite analogue of one variable generalization of entry 3}, we get
\begin{align*}
\sum_{n=1}^{N} \begin{bmatrix}
N\\
n
\end{bmatrix} \frac{(-1)^n (q)_{n-1} (b/a)_{n} a^{n} q^{n(n+1)/2} }{(b)_{n} (q/e)_{n} e^n} = \sum_{n=1}^N \begin{bmatrix}
N\\
n
\end{bmatrix} \frac{(q)_{n-1} (b)_{N-n}  b^n}{(b)_N } \left(\frac{(aq/be)_{n}}{(q/e)_{n}}-1 \right).
\end{align*}
Now divide both sides by $1-b/a$ in the above expression and then let $b \rightarrow a$ to get the result.

\end{proof}

\section{Proof of Theorem \ref{finite analogue of generalization of garvan's identity} and some of its corollaries}
In this section, our goal is to prove Theorem \ref{finite analogue of generalization of garvan's identity}. However, before proving this theorem, it is essential to introduce a few key results. These lemmas plays a vital role in the proof of our theorem.
\begin{lemma}\label{generalization of lemma 9.3 Untrodden}
For $N\in\mathbb{N}$, define
\begin{align*}
S_1(z,d,q,N):= \sum_{n=1}^N \begin{bmatrix}
N\\
n
\end{bmatrix}_{q^2} \frac{(q^2;q^2)_n (dq^2;q^2)_{n-1} (zq;q^2)_{N-n}(zq)^n}{(zq^2;q^2)_n(zq;q^2)_N},
\end{align*}
then
\begin{align*}
S_1(z,d,q,N)= \sum_{n=1}^N \begin{bmatrix}
N\\
n
\end{bmatrix}_{q^2} \frac{(q^2;q^2)_n (dq;q^2)_{n-1} (zq^2;q^2)_{N-n}z^nq^{2n-1}}{(zq;q^2)_n(zq^2;q^2)_N}.
\end{align*}
\end{lemma}

\begin{proof}
We start with
\begin{align*}
S_1(z,d,q,N)&= \sum_{n=1}^N \begin{bmatrix}
N\\
n
\end{bmatrix}_{q^2} \frac{(q^2;q^2)_n (dq^2;q^2)_{n-1} (zq;q^2)_{N-n}(zq)^n}{(zq^2;q^2)_n(zq;q^2)_N} \nonumber \\
&=\sum_{n=0}^{N-1} \frac{(q^2;q^2)_N (dq^2;q^2)_{n} (zq;q^2)_{N-n-1}(zq)^{n+1}}{(q^2;q^2)_{N-n-1}(zq^2;q^2)_{n+1}(zq;q^2)_N}\nonumber\\
&= \frac{zq(1-q^{2N})}{(1-zq^{2N-1})(1-zq^2)}\sum_{n=0}^{N-1}\frac{(q^2;q^2)_{N-1} (zq;q^2)_{N-n-1}(dq^2;q^2)_{n} (zq)^{n}}{(zq;q^2)_{N-1}(q^2;q^2)_{N-n-1}(zq^4;q^2)_{n}}.
\end{align*}
Now use \cite[p.~351,~Appendix (I.11)]{GasperRahman}
\begin{align}\label{q-series identity}
\frac{(b;q^2)_N(a;q^2)_{N-n}}{(a;q^2)_N(b;q^2)_{N-n}}\left(\frac{a}{b}\right)^n=\frac{(q^{2-2N}/b;q^2)_n}{(q^{2-2N}/a;q^2)_n}
\end{align}
with $a=zq$ and $b=q^2$ and $N$ replaced by $N-1$. We have
\begin{align}\label{s_1(z,d,q,N)}
S_1(z,d,q,N)&=\frac{zq(1-q^{2N})}{(1-zq^{2N-1})(1-zq^2)}\sum_{n=0}^{N-1}\frac{(q^{2-2N};q^2)_n(dq^2;q^2)_{n}}{(q^{3-2N}/z;q^2)_n(zq^4;q^2)_{n}}q^{2n}\nonumber\\
&=\frac{zq(1-q^{2N})}{(1-zq^{2N-1})(1-zq^2)}{}_{3}\phi_{2}\left[\begin{matrix} q^{2-2N}, & dq^2 , & q^2 \\ & zq^4, & \frac{q^{3-2N}}{z} \end{matrix} \, ; q^2, q^2  \right].
\end{align}
At this point, make use of \eqref{corollary of finite heine transformation} with $N$ and $q$ replaced by $N-1$ and $q^2$ respectively, and take $\alpha=dq^2,~\beta=q^2,~\gamma=zq^4$ and $\tau=zq$ in \eqref{s_1(z,d,q,N)} yields
\begin{align*}
S_1(z,d,q,N)&=\frac{zq(1-q^{2N})}{(1-zq^{2N-1})(1-zq^2)}\frac{(zq^2;q^2)_{N-1} (zq^3;q^2)_{N-1}}{(zq^4;q^2)_{N-1}(zq;q^2)_{N-1}} {}_{3}\phi_{2}\left[\begin{matrix} q^{2-2N}, & dq , & q^2 \\ & zq^3, & \frac{q^{2-2N}}{z} \end{matrix} \, ; q^2, q^2  \right]\\
&=\frac{zq(1-q^{2N})}{(1-zq)(1-zq^{2N})}\sum_{n=0}^{N-1}\frac{(q^{2-2N};q^2)_n (dq;q^2)_n}{(q^{2-2N}/z)_n(zq^3;q^2)_n}q^{2n}.
\end{align*}
Here again we use \eqref{q-series identity} with $N$ replace by $N-1,~a=zq^2$ and $b=q^2$, we get
\begin{align*}
S_1(z,d,q,N)&=\frac{zq(1-q^{2N})}{(1-zq)(1-zq^{2N})}\sum_{n=0}^{N-1}\frac{(q^2;q^2)_{N-1} (zq^2;q^2)_{N-n-1} (dq;q^2)_n}{ (zq^2;q^2)_{N-1} (q^2;q^2)_{N-n-1}(zq^3;q^2)_n}\left(zq^2\right)^n\\
&=\sum_{n=1}^{N}\begin{bmatrix}
N\\n
\end{bmatrix}_{q^2} \frac{(q^2;q^2)_n (zq^2;q^2)_{N-n} (dq;q^2)_{n-1}}{ (zq^2;q^2)_{N} (zq;q^2)_n}z^nq^{2n-1}.
\end{align*}
This proves the result.
\end{proof}

\begin{lemma}\label{generalization of lemma 9.4 Untrodden}
For $N\in\mathbb{N}$, we have
\begin{align*}
S_1(z,d,q,N)= \sum_{n=1}^N \begin{bmatrix}
N\\
n
\end{bmatrix}_{q^2} \left( \frac{(dq)_{2n-2} z^{2n-1} q^{n(2n-1)}}{(zq)_{2n-1}} + \frac{(dq)_{2n-1} z^{2n} q^{n(2n+1)}}{(zq)_{2n}} \right)\frac{(q^2;q^2)_n}{(dzq^{2N+1};q^2)_n}.
\end{align*}
\end{lemma}

\begin{proof} Replace $q$ by $q^2$ and let $\alpha=\frac{d}{q},~\beta=zq$ and $\tau=zq^2$ in the finite analogue of the Rogers-Fine identity \cite[Lemma 9.2]{DEMS} and then multiplying both sides of the resulting identity by $\frac{(1-dzq^{2N+3})}{(1-zq)(1-zq^{2N+2})}$, we get
\begin{align*}
&(1-dzq^{2N+3})\sum_{n=0}^N \frac{(q^2;q^2)_N (dq;q^2)_n (zq^2;q^2)_{N-n} z^n q^{2n} }{(q^2;q^2)_{N-n} (zq;q^2)_{n+1} (zq^2;q^2)_{N+1} } \\
&=\sum_{n=0}^N \frac{(q^2;q^2)_N (dq;q^2)_n (dq^2;q^2)_n (dzq^5;q^2)_{N} z^{2n} q^{3n+2n^2} (1-dzq^{4n+3})}{(q^2;q^2)_{N-n} (zq;q^2)_{n+1} (zq^2;q^2)_{n+1} (dzq^5;q^2)_{N+n}} \\
&=\sum_{n=0}^N \frac{(dq;q)_{2n} z^{2n} q^{n(2n+3)} (1-dzq^{4n+3})}{(zq;q)_{2n+2}}\frac{(q^2;q^2)_N (dzq^5;q^2)_{N}}{(q^2;q^2)_{N-n} (dzq^5;q^2)_{N+n}}\\
&=\sum_{n=0}^N \left[ \frac{(dq;q)_{2n} z^{2n} q^{n(2n+3)}}{(zq;q)_{2n+1}} + \frac{(dq;q)_{2n+1} z^{2n+1} q^{(n+2)(2n+1)}}{(zq;q)_{2n+2}} \right] \frac{(q^2;q^2)_N}{(q^2;q^2)_{N-n} (dzq^{2N+5};q^2)_{n}}.
\end{align*}
Now re-indexing both sums and then multiplying both sides by $zq(1-q^{2N+2})$ and further replace $N$ by $N-1$, we get
\begin{align*}
&\sum_{n=1}^{N} \frac{(q^2;q^2)_N (dq;q^2)_{n-1} (zq^2;q^2)_{N-n} z^{n} q^{2n-1}}{(q^2;q^2)_{N-n} (zq;q^2)_{n} (zq^2;q^2)_{N}}\\
&=\sum_{n=1}^{N} \left[ \frac{(dq;q)_{2n-2} z^{2n-2} q^{(n)(2n+1)}}{(zq;q)_{2n-1}} + \frac{(dq;q)_{2n-1} z^{2n} q^{(n)(2n+1)}}{(zq;q)_{2n}} \right] \frac{(q^2;q^2)_N}{(q^2;q^2)_{N-n} (dzq^{2N+3};q^2)_{n}}.
\end{align*}
The result follows from the above equation and Lemma \ref{generalization of lemma 9.3 Untrodden}.
\end{proof}

Now we are ready to prove Theorem \ref{finite analogue of generalization of garvan's identity}.

\begin{proof}[Theorem \ref{finite analogue of generalization of garvan's identity}][]
Take $e=1$, replace $q$ by $q^2$ then $z$ by $z/q$ and put $c=z$ in Corollary \ref{finite analogue of 4-variable generalization of Andrews Garvan Liang}. This yields
\begin{align*}
\sum_{n=1}^N \begin{bmatrix}
N\\
n
\end{bmatrix}_{q^2} \frac{(-1)^n(q^2;q^2)_n (z/d;q^2)_n (zd)^n q^{n^2}}{(z-d)(zq;q^2)_n(zq^2;q^2)_n}=  \sum_{n=1}^N \begin{bmatrix}
N\\
n
\end{bmatrix}_{q^2} \frac{(q^2;q^2)_n(zq^2;q^2)_{N-n}(dq;q^2)_{n-1}z^nq^{2n-1}}{(zq^2;q^2)_N(zq;q^2)_{n}}.
\end{align*}
Now the result follows from above equation and Lemma \ref{generalization of lemma 9.3 Untrodden} and making use of \ref{generalization of lemma 9.4 Untrodden}.
\end{proof}
Here are some immediate implications of Theorem \ref{finite analogue of generalization of garvan's identity}.

\begin{corollary}\label{generalization of corollary 9.5 untrodden}
For $N \in \mathbb{N}$, we have 
\begin{align*}
&\sum_{n=1}^N\begin{bmatrix}
N\\n 
\end{bmatrix}_{q^2} \frac{(-1)^{n-1}(q^2/d;q^2)_{n-1} d^{n-1} q^{n^2}}{(q;q^2)_n} = \sum_{n=1}^N\begin{bmatrix}
N\\n 
\end{bmatrix}_{q^2} \frac{(dq^2;q^2)_{n-1} (q;q^2)_{N-n} q^{n}}{(q;q^2)_N} \nonumber\\
& =\sum_{n=1}^N \frac{(dq;q^2)_{n-1} q^{2n-1}}{(q;q^2)_n} =\sum_{n=1}^N\begin{bmatrix}
N\\n 
\end{bmatrix}_{q^2} \frac{(q^2;q^2)_{n-1} (dq)_{2n-1}(1-dq^{4n-1})q^{n(2n-1)}}{(dq^{2N+1};q^2)_n (q)_{2n-1} (1-dq^{2n-1})}.
\end{align*}
\end{corollary}

\begin{proof}
Substituting $z=1$ in Theorem \ref{finite analogue of generalization of garvan's identity}, Lemma \ref{generalization of lemma 9.3 Untrodden} and \ref{generalization of lemma 9.4 Untrodden}, then combining the result together, we get the required result.
\end{proof}
Substituting $d=1$ in Corollary \ref{generalization of corollary 9.5 untrodden} gives \cite[Corollary 9.5]{DEMS}. Also, letting $N \rightarrow \infty$ in Corollary \ref{generalization of corollary 9.5 untrodden} gives a one variable generalization to \cite[p.~345,~Remark 3]{DM} stated below.

\begin{corollary} We have
\begin{align*}
\sum_{n=1}^\infty \frac{(-1)^{n}(1/d;q^2)_n d^n q^{n^2}}{(1-d)(q;q)_{2n}} &= \sum_{n=1}^\infty \frac{(dq^2;q^2)_{n-1} q^{n}}{(q^2;q^2)_n} =\sum_{n=1}^\infty \frac{(dq;q^2)_{n-1}q^{2n-1}}{(q;q^2)_n} \nonumber\\
&  =\sum_{n=1}^\infty \frac{(dq)_{2n-1}(1-dq^{4n-1})q^{n(2n-1)}}{ (q)_{2n-1} (1-q^{2n}) (1-dq^{2n-1})}.
\end{align*}
\end{corollary}
A further implication of Theorem \ref{finite analogue of generalization of garvan's identity} is given below.

\begin{corollary}\label{generalization of corollary 9.7 untrodden}
For $N\in\mathbb{N}$, we have
\begin{align*}
\frac{1}{(d-q)}\left(1-\frac{(dq^2;q^2)_N}{(q^3;q^2)_N} \right)=\sum_{n=1}^N\begin{bmatrix}
N\\n 
\end{bmatrix}_{q^2}\frac{(-1)^{n-1}(q/d;q^2)_n d^n q^{n(n+1)}}{(d-q)(q^3;q^2)_n}\nonumber\\=\sum_{n=1}^N\begin{bmatrix}
N\\n 
\end{bmatrix}_{q^2}\frac{(dq;q^2)_{n-1} (q^3;q^2)_{N-n} q^{3n-1}}{(q^3;q^2)_N}=\sum_{n=1}^N\frac{(dq^2;q^2)_{n-1}}{(q^3;q^2)_n}q^{2n}\nonumber \\=(1-q)\sum_{n=1}^N\begin{bmatrix}
N\\n 
\end{bmatrix}_{q^2}\frac{(q^2;q^2)_{n-1}(dq)_{2n-1}(1-dq^{4n})q^{2n^2+n-1}}{(dq^{2N+2};q^2)_n (q)_{2n-1} (1-dq^{2n-1}) (1-q^{2n+1})}.
\end{align*}
\end{corollary}

\begin{proof}
In \cite[Lemma 4.2]{DixitPatel}, replace $q$ by $q^2$  and then put $c=q$. We get
\begin{align*}
\sum_{n=1}^N \begin{bmatrix}
N\\
n
\end{bmatrix}_{q^2} \frac{(-1)^{n-1} (q/d;q^2)_n d^n q^{n(n+1)}}{(q^3;q^2)_n} = 1- \frac{(dq^2;q^2)_N}{(q^3;q^2)_N}.
\end{align*}
Now, substitute $z=q$ in Theorem \ref{finite analogue of generalization of garvan's identity}, Lemma \ref{generalization of lemma 9.3 Untrodden}, and \ref{generalization of lemma 9.4 Untrodden}. Combining all the resulting identities and use the above identity to get the desired result.
\end{proof}

Note that, by substituting $d=1$ in Corollary \ref{generalization of corollary 9.7 untrodden} gives \cite[Corollary 9.7]{DEMS}.

\section{Identities involving finite sum of ${}_3\phi_2$ and ${}_2\phi_1$}
In this section, our objective is to establish an identity that involves a finite sum of the basic hypergeometric series ${}_3\phi_2$ and ${}_2\phi_1$. This identity is significant as it effectively serves as a generalization of an identity \eqref{finite sum of 2phi1} of Dixit and Patel. Before going into the proof of this identity, we will begin by introducing and proving a few important lemmas. These lemmas will be essential to prove our theorem.
\begin{lemma}\label{generalization of lemma 4.1 Dixit-Patel}
We have
\begin{align}\label{identitiy involving finite sum of 2phi1}
&\sum_{n=1}^N \begin{bmatrix}
N\\
n
\end{bmatrix} \frac{(-1)^{n-1}(c/d)_n (q/e)_{n-1} e^{n-1} d^n q^{n(n+1)/2}}{(cq)_n (q)_{n-1}}\left(\sum_{k=1}^n\frac{q^k}{1-q^k}\right) \nonumber \\ 
& = \frac{(c/d)_\i (deq)_\i (q/e)_\i}{(cq)_\i (deq^{N+1})_\i (q)_\i} \sum_{k=1}^N \begin{bmatrix}
N\\
k
\end{bmatrix} \frac{e^{k-1}d^k q^{k(k+1)}}{(deq)_k(1-q^k)} \sum_{m=0}^\i \frac{(dq)_m (deq^{N+1})_m}{(deq^{k+1})_m (q)_m} \left( \frac{cq^k}{d} \right)^m \nonumber \\ 
& \hspace{5cm} \times {}_{2}\phi_{1}\left( \begin{matrix} e, & deq^{N+m+1}  \\ & deq^{k+m+1} \end{matrix}  ;~ q,~ \frac{q^k}{e} \right).
\end{align}
\end{lemma}

\begin{proof}
Using van Hamme's identity \cite{hamme}
\begin{align*}
\sum_{k=1}^n\frac{q^k}{1-q^k}=\sum_{k=1}^n \begin{bmatrix}
n\\
k
\end{bmatrix} \frac{(-1)^{k-1} q^{k(k+1)/2} }{1-q^k},
\end{align*}
the left hand side expression of \eqref{identitiy involving finite sum of 2phi1} becomes
\begin{align*}
&\sum_{n=1}^N \begin{bmatrix}
N\\
n
\end{bmatrix} \frac{(-1)^{n-1}(c/d)_n (q/e)_{n-1} e^{n-1} d^n q^{n(n+1)/2}}{(cq)_n (q)_{n-1}}\sum_{k=1}^n \begin{bmatrix}
n\\
k
\end{bmatrix} \frac{(-1)^{k-1} q^{k(k+1)/2} }{1-q^k} \\
&=\sum_{k=1}^N \sum_{n=k}^N \begin{bmatrix}
N\\
n
\end{bmatrix} \begin{bmatrix}
n\\
k
\end{bmatrix} \frac{(-1)^{n-1}(c/d)_n (q/e)_{n-1} e^{n-1} d^n q^{n(n+1)/2}}{(cq)_n (q)_{n-1}}  \frac{(-1)^{k-1} q^{k(k+1)/2} }{1-q^k}.
\end{align*}
Now re-index the second sum of the above expression by replacing $n$ by $n+k$ yields
\begin{align*}
&(q)_N\sum_{k=1}^N \frac{(-1)^{k-1} q^{k(k+1)/2} }{(q)_k(1-q^k)} \sum_{n=0}^{N-k} \frac{(-1)^{n+k-1}(c/d)_{n+k} (q/e)_{n+k-1} e^{n+k-1} d^{n+k} q^{(n+k)(n+k+1)/2}}{(cq)_{n+k} (q)_n (q)_{N-n-k} (q)_{n+k-1}} \\
&= (q)_N\sum_{k=1}^N \frac{(c/d)_k (q/e)_{k-1} q^{k(k+1)} e^{k-1} d^k }{(cq)_k(q)_k (q)_{k-1} (1-q^k)} \sum_{n=0}^{N-k} \frac{(-1)^{n}(cq^k/d)_{n} (q^k/e)_{n} (de)^{n} q^{\frac{n(n+1)}{2}+nk}}{(cq^{k+1})_{n} (q)_n (q)_{N-n-k} (q^k)_{n}}.
\end{align*}
Taking $x=1$ and replacing $N$ by $N-k$ in \eqref{basic formula 1} and making use of it in the second sum of the above expression to obtain
\begin{align*}
&(q)_N\sum_{k=1}^N \frac{(c/d)_k (q/e)_{k-1} q^{k(k+1)} e^{k-1} d^k }{(cq)_k(q)_k (q)_{k-1} (1-q^k)} \sum_{n=0}^{N-k} \frac{ \left(q^{-(N-k)} \right)_n (cq^k/d)_{n} (q^k/e)_{n}  }{(cq^{k+1})_{n} (q)_n (q)_{N-k} (q^k)_{n}}\left(deq^{(N+1)}\right)^n \\
&=\sum_{k=1}^N \begin{bmatrix}
N\\
k
\end{bmatrix} \frac{(c/d)_k (q/e)_{k-1} q^{k(k+1)} e^{k-1} d^k }{(cq)_k (q)_{k-1} (1-q^k)} \frac{(cq^k/d)_\i}{(cq^{k+1})_\i} \sum_{n=0}^{N-k} \frac{ \left(q^{-(N-k)} \right)_n (q^k/e)_{n}  }{(q)_n (q^k)_{n}}\left(deq^{(N+1)}\right)^n \frac{(cq^{k+n+1})_\i}{(cq^{k+n}/d)_\i}.
\end{align*}
Now make use of $q$-binomial theorem \eqref{q-Binomial theorem} and get
\begin{align*}
&\frac{(c/d)_\i}{(cq)_\i}\sum_{k=1}^N \begin{bmatrix}
N\\
k
\end{bmatrix} \frac{(q/e)_{k-1} q^{k(k+1)} e^{k-1} d^k }{(q)_{k-1} (1-q^k)}  \sum_{n=0}^{N-k} \frac{ \left(q^{-(N-k)} \right)_n (q^k/e)_{n}  }{(q)_n (q^k)_{n}}\left(deq^{(N+1)}\right)^n\sum_{m=0}^\i \frac{(dq)_m}{(q)_m}\left(\frac{cq^{k+n}}{d}\right)^m \\
&=\frac{(c/d)_\i}{(cq)_\i}\sum_{k=1}^N \begin{bmatrix}
N\\
k
\end{bmatrix} \frac{(q/e)_{k-1} q^{k(k+1)} e^{k-1} d^k }{(q)_{k-1} (1-q^k)} \sum_{m=0}^\i \frac{(dq)_m}{(q)_m}\left(\frac{cq^{k}}{d}\right)^m {}_{2}\phi_{1}\left( \begin{matrix} q^{-(N-k)}, & q^k/e \\ &  q^k  \end{matrix}  ;~ q,~ deq^{N+m+1} \right).
\end{align*}
Now apply Heine's transformation \eqref{heine transformation} with $a=q^{-(N-k)},~b=q^k/e,~c=q^k$ and $z=deq^{N+m+1}$ and get required right hand side of \eqref{identitiy involving finite sum of 2phi1} upon simplification.

\end{proof}

\begin{lemma}\label{generalization to lemma 4.2 Dixit-Patel}
We have
\begin{align*}
\sum_{n=1}^N \begin{bmatrix}
N\\
n
\end{bmatrix} \frac{(-1)^{n-1}(c/d)_n (q/e)_{n-1} e^{n-1} d^n q^{n(n+1)/2}}{(cq)_n (q)_{n-1}} = \frac{(deq)_N}{(cq)_N} \sum_{n=1}^N \frac{(q^{-N})_n (d/c)_n (eq)_{n-1}}{(deq)_n (q)_n (q)_{n-1}} \left(cq^{N+1}\right)^n.
\end{align*}
\end{lemma}

\begin{proof}
Putting $z=1$ in Corollary \ref{finite analogue of 4-variable generalization of Andrews Garvan Liang}, we have
\begin{align*}
\sum_{n=1}^N \begin{bmatrix}
N\\
n
\end{bmatrix} \frac{(-1)^{n}(c/d)_n (q/e)_{n-1} q^{n(n+1)/2}  d^n e^{n-1}}{(cq)_n(q)_{n-1}}
&= \frac{1}{(cq)_N} \sum_{n=1}^N \begin{bmatrix}
N\\
n
\end{bmatrix} \frac{(ceq)_{N-n} (d/c)_{n}(q/e)_{n-1}(cq)^ne^{n-1}}{(q)_{n-1}} \\
&=\frac{(ceq)_N}{e(cq)_N} \sum_{n=1}^N \frac{(q)_N (ceq)_{N-n} (d/c)_{n}(q/e)_{n-1}(ceq)^n}{(q)_{N-n} (ceq)_N (q)_n (q)_{n-1}}.
\end{align*}
Utilizing \eqref{basic formula 2} with $x=ce$ in right hand side of the above expression, we get
\begin{align*}
\sum_{n=1}^N \begin{bmatrix}
N\\
n
\end{bmatrix} &\frac{(-1)^{n}(c/d)_n (q/e)_{n-1} q^{n(n+1)/2}  d^n e^{n-1}}{(cq)_n(q)_{n-1}}= \frac{(ceq)_N}{e(cq)_N} \sum_{n=1}^N \frac{\left(q^{-N}\right)_n (d/c)_{n}(q/e)_{n-1}(q)^n}{ \left( q^{-N}/ce \right)_n (q)_n (q)_{n-1}} \\
&=\frac{q}{e} \frac{(ceq)_N}{(cq)_N} \frac{\left(1-q^{-N}\right) (1-d/c)}{\left( 1-q^{-N}/ce\right) (1-q)} {}_{3}\phi_{2}\left[ \begin{matrix} q^{-(N-1)},& q/e, & dq/c \\ q^2 & q^{-(N-1)}/ce & \end{matrix}  ;~ q,~ q \right].
\end{align*}
Applying ${}_{3}\phi_{2}$ transformation with $N$ replaced by $N-1$ and substituting $B=q/e,~C=dq/c,~D=q^2$ and $E=q^{-(N-1)}/ce$ in \eqref{3_phi_2 finite}, we get
\begin{align*}
&\sum_{n=1}^N \begin{bmatrix}
N\\
n
\end{bmatrix} \frac{(-1)^{n}(c/d)_n (q/e)_{n-1} q^{n(n+1)/2}  d^n e^{n-1}}{(cq)_n(q)_{n-1}} \\
&=\frac{q}{e} \frac{(ceq)_N}{(cq)_N} \frac{\left(1-q^{-N}\right) (1-d/c)}{\left( 1-q^{-N}/ce\right) (1-q)} \frac{\left( q^{-N}/de \right)_{N-1}}{\left( q^{-(N-1)}/ce \right)_{N-1}} \left(\frac{dq}{c} \right)^{N-1} {}_{3}\phi_{2}\left[ \begin{matrix} q^{-(N-1)}, & dq/c, & eq \\ q^2 & deq^2 & \end{matrix}  ;~ q,~ cq^{N+1} \right]\\
&=\frac{q}{e} \frac{(ceq)_N}{(cq)_N} \frac{\left(1-q^{-N}\right) (1-d/c)}{\left( 1-q^{-N}/ce\right) (1-q)} \frac{(deq^2)_{N-1}}{(ceq)_{N-1}} \sum_{n=0}^{N-1} \frac{\left(q^{-(N-1)} \right)_n (dq/c)_n (eq)_n}{ (q^2)_n (deq^2)_n (q)_n}\left(cq^{N+1}\right)^n \\
&=\frac{(deq)_N}{(cq)_N}\sum_{n=0}^{N-1} \frac{\left(q^{-N} \right)_{n+1} (d/c)_{n+1} (eq)_n}{ (q)_{n+1} (deq)_{n+1} (q)_n}\left(cq^{N+1}\right)^{n+1},
\end{align*}
this completes the proof of Lemma \ref{generalization to lemma 4.2 Dixit-Patel} by re-indexing the above sum by replacing $n$ by $n-1$.
\end{proof}

We now ready to prove our main identity of this section, which is a one variable generalization of an identity \eqref{finite sum of 2phi1} of Dixit and Patel.

\begin{theorem}\label{generalization of Theorem 4.3 Dixit-Patel}
We have
\begin{align}\label{generalization of Theorem 4.3 Dixit-Patel equation}
&\sum_{n=1}^N \begin{bmatrix}
N\\
n
\end{bmatrix} \frac{n(-1)^{n-1}(c/d)_n (q/e)_{n-1} e^{n-1} d^n q^{n(n+1)/2}}{(cq)_n (q)_{n-1}} + \frac{(c/d)_\i (deq)_\i (q/e)_\i}{(cq)_\i (deq^{N+1})_\i (q)_\i} \nonumber \\
& \times \sum_{k=1}^N \begin{bmatrix}
N\\
k
\end{bmatrix} \frac{e^{k-1}d^k q^{k(k+1)}}{(deq)_k(1-q^k)} \sum_{m=0}^\i \frac{(dq)_m (deq^{N+1})_m}{(deq^{k+1})_m (q)_m} \left( \frac{cq^k}{d} \right)^m {}_{2}\phi_{1}\left( \begin{matrix} e, & deq^{N+m+1}\\ & deq^{k+m+1} \end{matrix}  ;~ q,~ \frac{q^k}{e} \right) \nonumber \\
&=\frac{c}{c-d} \frac{(deq)_N}{(cq)_N} \sum_{n=1}^N \frac{(q^{-N})_n (d/c)_n (eq)_{n-1}}{(deq)_n (q)_n (q)_{n-1}} \left(cq^{N+1}\right)^n + \frac{(q/e)_{N-1}}{(cq)_N (q)_{N-1}} \nonumber \\
& \times \sum_{k=1}^N \begin{bmatrix}
N\\
k
\end{bmatrix} \frac{(cq/d)_k (deq)_{N-k} (dq)^ke^{k-1}}{ (1-q^k)} {}_{3}\phi_{2}\left( \begin{matrix} q^{-(N-k)}, & e, & ceq  \\ deq & eq^{1-N} & \end{matrix}  ;~ q,~ q \right).
\end{align}
\end{theorem}
\begin{remark}
It can be easily observe that, $e=1$ in the above identity gives \eqref{finite sum of 2phi1}.
\end{remark}

\begin{proof}[Theorem \ref{generalization of Theorem 4.3 Dixit-Patel}][]
Differentiating both sides of Corollary \ref{finite analogue of 4-variable generalization of Andrews Garvan Liang} with respect to $z$ and then put $z=1$ to get
\begin{align}\label{define T_1+T_2}
\sum_{n=1}^N \begin{bmatrix}
N\\
n
\end{bmatrix} \frac{(-1)^{n-1}(c/d)_n (q/e)_{n-1} e^{n-1} d^n q^{n(n+1)/2}}{(cq)_n (q)_{n-1}}\left(n+\sum_{k=1}^n\frac{q^k}{1-q^k}\right) =:T_1+T_2,
\end{align}
where
\begin{align}
T_1&=\frac{(d-c)}{c(cq)_N} \sum_{n=1}^N \begin{bmatrix}
N\\
n
\end{bmatrix} \frac{(ceq)_{N-n} (dq/c)_{n-1}(q/e)_{n-1}(cq)^ne^{n-1}}{(q)_{n-1}}, \label{define T_1} \\
T_2&=\frac{(d-c)}{c(cq)_N} \sum_{n=1}^N \begin{bmatrix}
N\\
n
\end{bmatrix} \frac{(ceq)_{N-n} (dq/c)_{n-1}(q/e)_{n-1}(cq)^ne^{n-1}}{(q)_{n-1}}\left(\sum_{k=1}^n\frac{q^k}{1-q^k} - \sum_{k=1}^{n-1}\frac{q^kd/c}{1-q^kd/c} \right). \label{define T_2}
\end{align}
Using Corollary \ref{finite analogue of 4-variable generalization of Andrews Garvan Liang} with $z=1$ and Lemma \ref{generalization to lemma 4.2 Dixit-Patel} in \eqref{define T_1}, we get
\begin{align}\label{final value of T_1}
T_1=\frac{(deq)_N}{(cq)_N} \sum_{n=1}^N \frac{(q^{-N})_n (d/c)_n (eq)_{n-1}}{(deq)_n (q)_n (q)_{n-1}} \left(cq^{N+1}\right)^n.
\end{align}
Guo and Zhang \cite[Corollary~3.1]{guozhang} proved the following identity for $n\geq 0$ and $0\leq m \leq n$:
\begin{align*}
\sum_{\substack{k=0\\ k\neq m}}^n \begin{bmatrix}
n\\
k
\end{bmatrix} \frac{(q/x)_k (x)_{n-k}}{1-q^{k-m}}x^k = (-1)^mq^{\frac{m(m+1)}{2}} \begin{bmatrix}
n\\
m
\end{bmatrix} \left( xq^{-m}\right)_n \left( \sum_{k=0}^{n-1} \frac{xq^{k-m}}{1-xq^{k-m}} - \sum_{\substack{k=0\\ k\neq m}}^n \frac{q^{k-m}}{1-q^{k-m}} \right).
\end{align*}
Substituting $m=0$ and $x=d/c$ in the above identity gives
\begin{align*}
\sum_{k=1}^n \frac{q^{k}}{1-q^{k}} - \sum_{k=1}^{n-1} \frac{q^{k}d/c}{1-q^{k}d/c} = \frac{d}{c-d} - \frac{1}{(d/c)_n} \sum_{k=1}^n \begin{bmatrix}
n\\
k
\end{bmatrix} \frac{(cq/d)_k (d/c)_{n-k}}{1-q^{k}}\left(\frac{d}{c}\right)^k.
\end{align*}
Use the above expression and employing \eqref{define T_1} in \eqref{define T_2} yields
\begin{align}\label{T_2 in term of T_1 and T_3}
T_2=\frac{d}{c-d}T_1+T_3,
\end{align}
where
\begin{align*}
T_3&=\frac{1}{(cq)_N} \sum_{n=1}^N \begin{bmatrix}
N\\
n
\end{bmatrix} \frac{(ceq)_{N-n} (q/e)_{n-1}(cq)^ne^{n-1}}{ (q)_{n-1}} \sum_{k=1}^n \begin{bmatrix}
n\\
k
\end{bmatrix} \frac{(cq/d)_n (d/c)_{n-k} (d/c)^k}{1-q^{k}} \\
&=\frac{(q)_N}{(cq)_N} \sum_{n=1}^N \frac{(ceq)_{N-n} (q/e)_{n-1}(cq)^ne^{n-1}}{ (q)_{N-n} (q)_{n-1}} \sum_{k=1}^n \frac{(cq/d)_k (d/c)_{n-k} (d/c)^k}{(q)_k (q)_{n-k} (1-q^{k})} \\
&=\frac{(q)_N}{(cq)_N} \sum_{k=1}^N \frac{(cq/d)_k  (d/c)^k}{(q)_k  (1-q^{k})} \sum_{n=k}^N \frac{(ceq)_{N-n} (d/c)_{n-k} (q/e)_{n-1}(cq)^ne^{n-1}}{ (q)_{n-k} (q)_{N-n} (q)_{n-1}}\\
&=\frac{(q)_N}{(cq)_N} \sum_{k=1}^N \frac{(cq/d)_k  (d/c)^k}{(q)_k  (1-q^{k})} \sum_{n=0}^{N-k} \frac{(ceq)_{N-n-k} (d/c)_{n} (q/e)_{n+k-1}(cq)^{n+k}e^{n+k-1}}{ (q)_{n} (q)_{N-n-k} (q)_{n+k-1}}\\
&=\frac{(q)_N}{(cq)_N} \sum_{k=1}^N \frac{(cq/d)_k (q/e)_{k-1} (dq)^k e^{k-1}}{(q)_k (q)_{k-1} (1-q^{k})} \sum_{n=0}^{N-k} \frac{(ceq)_{N-n-k} (d/c)_{n} (q^k/e)_{n}(ceq)^{n}}{ (q)_{N-n-k} (q)_{n}  (q^k)_{n}} \\
&=\frac{(q)_N}{(cq)_N} \sum_{k=1}^N \frac{(cq/d)_k (ceq)_{N-k} (q/e)_{k-1} (dq)^k e^{k-1}}{(q)_k (q)_{N-k} (q)_{k-1} (1-q^{k})} \sum_{n=0}^{N-k} \frac{ \left(q^{-(N-k)}\right)_n (d/c)_{n} (q^k/e)_{n}q^{n}}{ (q^k)_{n} \left(q^{-(N-k)}/ce\right)_n (q)_{n}},
\end{align*}
where in the last step, we used \eqref{basic formula 2} with $N$ replaced by $N-k$ and $x=ce$. Now use \eqref{finite Heine transformation} with $N$ replaced by $N-k,~ \a=d/c,~ \b=q^k/e,~ \g=q^k$ and $\t=ceq$ to see that
\begin{align*}
\sum_{n=0}^{N-k} \frac{ \left(q^{-(N-k)}\right)_n (d/c)_{n} (q^k/e)_{n}q^{n}}{ (q^k)_{n} \left(q^{-(N-k)}/ce\right)_n (q)_{n}}=\frac{(q^k/e)_{N-k} (deq)_{N-k}}{(q^k)_{N-k} (ceq)_{N-k}} {}_{3}\phi_{2}\left[\begin{matrix} q^{-(N-k)}, & e, & ceq \\  deq, & eq^{1-N} \end{matrix} \, ; q, q  \right].
\end{align*}
Using the above identity in the final expression of $T_3$, we get
\begin{align}\label{final value of T_3}
T_3&=\frac{(q/e)_{N-1}}{(cq)_N(q)_{N-1}} \sum_{k=1}^N \begin{bmatrix}
N\\k
\end{bmatrix} \frac{(cq/d)_k (deq)_{N-k} (dq)^k e^{k-1}}{(1-q^{k})}{}_{3}\phi_{2}\left[\begin{matrix} q^{-(N-k)}, & e, & ceq \\  deq, & eq^{1-N} \end{matrix} \, ; q, q  \right].
\end{align}
Now utilize \eqref{final value of T_1} and \eqref{final value of T_3} in \eqref{T_2 in term of T_1 and T_3} to deduce
\begin{align}\label{final value of T_2}
T_2&=\frac{d}{c-d}\frac{(deq)_N}{(cq)_N} \sum_{n=1}^N \frac{(q^{-N})_n (d/c)_n (eq)_{n-1}}{(deq)_n (q)_n (q)_{n-1}} \left(cq^{N+1}\right)^n + \frac{(q/e)_{N-1}}{(cq)_N(q)_{N-1}} \nonumber \\
& \times \sum_{k=1}^N \begin{bmatrix}
N\\k
\end{bmatrix} \frac{(cq/d)_k (deq)_{N-k} (dq)^k e^{k-1}}{(1-q^{k})}{}_{3}\phi_{2}\left[\begin{matrix} q^{-(N-k)}, & e, & ceq \\  deq, & eq^{1-N} \end{matrix} \, ; q, q  \right].
\end{align}
Now combine \eqref{final value of T_1}, \eqref{final value of T_2} and Lemma \ref{generalization of lemma 4.1 Dixit-Patel} in \eqref{define T_1+T_2} to get the desired identity \eqref{generalization of Theorem 4.3 Dixit-Patel equation}.
\end{proof}

\section*{Acknowledgement}
The author wants to thank Dr. Bibekananda Maji for carefully reading this article and giving insightful suggestions. The author also acknowledges the University Grants Commission (UGC), India (ID: 191620076670), for awarding the Ph.D. fellowship, and expresses thanks to IIT Indore for providing state-of-the-art research facilities.

\end{document}